\documentclass[twoside]{article}
\usepackage{amsfonts,amsmath,amsthm,amstext,amssymb,hyperref,fancyhdr}
\usepackage{mathrsfs}
\usepackage{enumerate}
\usepackage{lipsum}
\usepackage{graphicx}
\usepackage[caption=false]{subfig}
\usepackage{algorithmic}

\newtheorem{theorem}{Theorem}[section]
\newtheorem{lemma}[theorem]{Lemma}
\newtheorem{corollary}[theorem]{Corollary}

\newtheorem{remark}{Remark}[section]
\numberwithin{theorem}{section}
\numberwithin{figure}{section}

\numberwithin{equation}{section}

\newcommand{\TheTitle}{FEM and CIP-FEM for Helmholtz Equation with PML}
\newcommand{\TheAuthors}{Yonglin Li and Haijun Wu}

\title{FEM and CIP-FEM for Helmholtz Equation \\with High Wave Number and PML truncation}
\author{
  Yonglin Li\thanks{Department of Mathematics, Nanjing University, Jiangsu, 210093, People's Republic of China ({\sf liyonglin@smail.nju.edu.cn,  hjw@nju.edu.cn}). This work was partially supported by the NSF of China under grants 11525103, 91630309,  and 11621101.}
  \and
  Haijun Wu\footnotemark[1]
}

\usepackage{amsopn}

\usepackage{color}

\newcommand{\abs}[1]{\left\vert #1 \right\vert}

\newcommand{\He}[1]{\left\Vert{\hskip -2.7pt}\left\vert #1 \right\vert{\hskip -2.7pt}\right\Vert}
\newcommand{\he}[1]{\big\vert\kern-0.25ex\big\vert\kern-0.25ex\big\vert #1 \big\vert\kern-0.25ex\big\vert\kern-0.25ex\big\vert}

\newcommand{\Lt}[2]{\left\Vert #1\right\Vert_{L^2(#2)}}
\newcommand{\Ho}[2]{\left\Vert #1\right\Vert_{H^1(#2)}}
\newcommand{\sHo}[2]{\left\vert #1\right\vert_{H^1(#2)}}
\newcommand{\Ht}[2]{\left\Vert #1\right\Vert_{H^2(#2)}}
\newcommand{\sHt}[2]{\left\vert #1\right\vert_{H^2(#2)}}
\newcommand{\inD}[2]{(#1,#2)_{\mathcal{D}}}
\newcommand{\ine}[2]{\langle #1,#2 \rangle_e}
\newcommand{\inOm}[2]{(#1,#2)_{\Omega}}
\newcommand{\hGnorm}[1]{\left\Vert #1\right\Vert_{H^{1/2}(\hat{\Gamma})}}
\newcommand{\Gnorm}[1]{\left\Vert #1\right\Vert_{H^{1/2}(\Gamma)}}

\newcommand{\ah}{a_{h}}
\newcommand{\bh}{b_{h}}
\newcommand{\Ih}{I_{h}}
\newcommand{\Pha}{P_{h}^{+}}
\newcommand{\Phm}{P_{h}^{-}}
\newcommand{\Ph}{P_{h}^{\pm}}

\newcommand{\B}{\mathcal{B}}
\newcommand{\D}{\mathcal{D}}
\newcommand{\R}{\mathbb{R}}
\newcommand{\Z}{\mathbb{Z}}
\newcommand{\N}{\mathbb{N}}
\newcommand{\C}{\mathbb{C}}
\newcommand{\EhI}{\mathcal{E}_h^I}
\newcommand{\Mh}{\mathcal{M}_h}

\newcommand{\J}{J_{n}}
\newcommand{\Jv}{J_{\nu}}

\newcommand{\Jz}{J_{0}}

\renewcommand{\H}{H_n^{(1)}}
\newcommand{\Hv}{H_{\nu}^{(1)}}

\newcommand{\Hz}{H_{0}^{(1)}}

\newcommand{\hu}{\hat{u}}
\newcommand{\hR}{\hat{R}}
\newcommand{\hC}{\hat{C}}

\newcommand{\bR}{\bar{R}}
\newcommand{\tr}{\tilde{r}}
\newcommand{\tu}{\tilde{u}}
\newcommand{\tx}{\tilde{x}}
\renewcommand{\tt}{\tilde{t}}
\newcommand{\thR}{\tilde{\hat{R}}}
\newcommand{\Om}{\Omega}
\newcommand{\hOm}{\hat{\Omega}}
\newcommand{\hGamma}{\hat{\Gamma}}

\newcommand{\intRhR}{\int_R^{\hR}}
\newcommand{\intrR}{\int_r^R}
\newcommand{\intrhR}{\int_r^{\hR}}
\newcommand{\intr}{\int_0^{r}}
\newcommand{\intR}{\int_0^R}
\newcommand{\inthR}{\int_0^{\hR}}

\newcommand{\al}{\alpha}
\newcommand{\be}{\beta}
\newcommand{\si}{\sigma}
\newcommand{\siz}{\sigma_0}
\newcommand{\de}{\delta}
\newcommand{\lam}{\lambda}
\newcommand{\na}{\nabla}
\newcommand{\Ga}{\Gamma}
\newcommand{\pa}{\partial}
\newcommand{\vp}{\varphi}
\newcommand{\supp}{\mathrm{supp}\, }
\renewcommand{\l}{\ell}
\renewcommand{\i}{{\rm\mathbf i}}
\newcommand{\sgn}{\mathrm{sgn}\,}
\newcommand{\OchO}{\Omega\cup\hat{\Omega}}

\newcommand*{\QED}{\hfill\ensuremath{\square}}
\newcommand{\diam}{\mathrm{diam}\,}
\newcommand{\ls}{\lesssim}
\newcommand{\ufem}{u_h^{\rm FEM}}

\begin{document}
\pagestyle{myheadings}
\markboth{\TheAuthors}{\TheTitle}
\date{}
\maketitle

\begin{abstract}
  The Helmholtz scattering problem with high wave number is truncated by the perfectly matched layer (PML) technique and then discretized by the linear continuous interior penalty finite element method (CIP-FEM). It is proved that the truncated PML problem satisfies the inf--sup condition with inf--sup constant of order $O(k^{-1})$. Stability and convergence of the truncated PML problem are discussed. In particular, the convergence rate is twice of the previous result. The preasymptotic error estimates in the energy norm of the linear CIP-FEM as well as FEM are proved to be $C_1kh+C_2k^3h^2$ under the mesh condition that $k^3h^2$ is sufficiently small. Numerical tests are provided to illustrate the preasymptotic error estimates and show that the penalty parameter in the CIP-FEM may be tuned to reduce greatly the pollution error.
\end{abstract}



\section{Introduction}\label{sec:Introduction}
In this paper, we consider the following acoustic scattering problem in $\R^d~(d=1,2,3)$,
\begin{alignat}{2}
-\Delta u - k^2 u &= f\qquad & \text{in } \R^{d}, \label{eq:Helm}\\
\abs{\frac{\partial u}{\partial r} - \i ku} &= o(r^{\frac{1-d}2})\qquad & r\to \infty, \label{eq:Somm}
\end{alignat}
which is going to be truncated into a bounded computational domain by the PML technique \cite{berenger1994perfectly,collino1998perfectly} and then discretized by the CIP-FEM \cite{Douglas1976Interior,Wu2013Pre,zhu2013preasymptotic,du2015preasymptotic} as well as the FEM. Here $r = \abs{x}$, $f\in L^2(\Om)$. Suppose $\supp f \subset \Omega:= \B(R)$ the ball with center at the origin and radius $R$. Denote by $\Gamma = \partial\Omega$.  Since we are considering the high wave number problems, we assume that $k\gg 1$.

The Helmholtz equation with large wave number is highly indefinite, which makes the analysis of its discretizations such as FEM  very difficult. For the linear FEM, the traditional technique, i.e., the duality argument (or Schatz argument, see \cite{aziz1979scattering, dss94,sch74}) gives merely the error estimate $\he{u-\ufem}_{\Om}\le Ckh$  under the mesh condition that $k^2h$ is small enough, but it is too strict for large $k$. Here $h$ is the mesh size and $\he{\cdot}_{\Om}:=\big(\Lt{\na\cdot}{\Om}^2+\Lt{\cdot}{\Om}^2\big)^\frac12$ denotes the energy norm. Ihlenburg and Babu\v{s}ka \cite{ib95a} considered the one dimensional problem discretized on equidistant grids, and proved the error estimate  $\he{u-\ufem}_{\Om}\le C_1kh+C_2k^3h^2$  under the condition that $kh$ less than some constant less than $\pi$. Note that the error bound includes two terms. The first term is of the same order as the interpolation error. The second term is bounded by the first one if $k^2h$ is small, but it dominates when $k^2h$ is large, which is called the pollution error in such a case \cite{bs00,ihlenburg98,ib95a}. We recall that the term “asymptotic error estimate” refers to the error estimate without pollution error and the term “preasymptotic error estimate” refers to the estimate with nonnegligible pollution effect. Recently, Wu \cite{Wu2013Pre} proved the same preasymptotic error estimate as above for higher dimensional problems while on unstructured meshes under the condition that $k^3h^2$ is sufficiently small. For error analyses of higher order FEM, we refer to \cite{du2015preasymptotic,ihlenburg1997finite,melenk2010convergence,melenk2011wavenumber,zhu2013preasymptotic}. 

The CIP-FEM, which was first proposed by Douglas and Dupont \cite{Douglas1976Interior} for elliptic and parabolic problems in 1970's, uses the same approximation space as the FEM but modifies the bilinear form of the FEM by adding a least squares term penalizing the jump of the normal derivative of the discrete solution at mesh interfaces. Recently the CIP-FEM has shown great potential in solving the Helmholtz problem with large wave number \cite{Wu2013Pre,zhu2013preasymptotic,du2015preasymptotic,bzw16,cwx15}. It is absolute stable if the penalty parameters are chosen as complex numbers with negative imaginary parts, it satisfies an error bound no larger than that of the FEM under the same mesh condition, its penalty parameters may be tuned to greatly reduce the pollution error, and so on. For preasymptotic and asymptotic error analyses of other methods including discontinuous Galerkin methods and spectral methods, we refer to \cite[etc.]{clx13,dgmz12,fw09,fw11,mps13,sw07}. 
We would like to mention that most error analyses in the literature including the above references are for the Helmholtz equation \eqref{eq:Helm} with the impedance boundary condition or the DtN boundary condition instead of the Sommerfeld radiation condition \eqref{eq:Somm}.

A more popular mesh termination technique for wave scattering problems is the PML method, which was originally proposed by Berenger \cite{berenger1994perfectly}. The key idea of the PML technique is to surround the computational domain $\Om$ by a special designed layer (as depicted in Figure \ref{fig:pml}) which  can exponentially  absorb all the outgoing waves entering the layer; even if the waves reflect off the truncated boundary $\hat\Gamma$, the returning waves after one round trip through the absorbing layer are very tiny \cite{berenger1994perfectly,berenger1996three,chew19943d,collino1998perfectly,ty98}. In fact, the fundamental analysis \cite[etc.]{lassas1998existence,chen2003adaptive,hohage2003solving,chen2005,bao2005,bramble2007analysis,cc08,blw10,cz10,cw12} indicates that the classical PML converges exponentially with perfectly non-reflection when the width of the layer or the PML parameter tends to infinity. In practice, the optimal choice of the parameters is important when the wave number is high \cite{cp99}.
\begin{figure}[htbp]
\centering
\includegraphics[width=6cm]{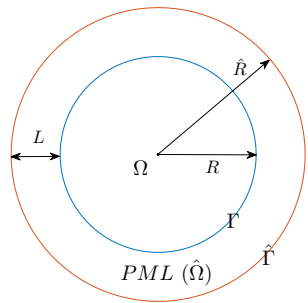}
\caption{Setting of the PML problem.}
\label{fig:pml}
\end{figure}
The truncated PML problem for \eqref{eq:Helm}--\eqref{eq:Somm} can be formulated as: Find $\hu\in H_0^1(\D)$ such that
\begin{equation}\label{PML}
 a(\hu,v) = \inD{f}{v} \quad\forall v\in H_0^1(\D)
\end{equation}
where $\D=\Om\cup\Ga\cup\hat\Om$ and $a(\cdot,\cdot)$ is defined in \eqref{a}. See section \ref{ss:PML} for details.

The purposes of this paper are twofold. First we truncate the Helmholtz problem \eqref{eq:Helm}--\eqref{eq:Somm} by PML and prove the inf--sup condition and the regularity estimate with explicit dependence on the wave number $k$ for the truncated PML problem. Secondly we discretize the truncated PML problem by the linear CIP-FEM (including the linear FEM) and derive the preasymptotic error estimates.
In \cite{chandler2008wave}  Chandler-Wilde and Monk have shown for the problem of acoustic scattering   from a star-shaped scatterer with  the DtN boundary condition that the inf--sup constant is of order $O(k^{-1})$. Melenk and Sauter \cite{melenk2010convergence} have proved that the solution $u$ to \eqref{eq:Helm}--\eqref{eq:Somm} satisfies the stability estimates $\|u\|_{H^j(\Om)}\le C k^{j-1} \|f\|_{L^2(\Om)}$ for $j=0,1,2$.
 While for the truncated PML problem, the best estimate in the literature is from Chen and Xiang \cite{chen2013source}, in which it is shown that the inf--sup constant is of order $O(k^{-\frac32})$. In this paper, we show that the inf--sup constant for the truncated PML problem is still of order $O(k^{-1})$, i.e., 
 \begin{align}\label{infsup}
 \frac{c_1}k \le\inf_{0\neq u\in H_0^1(\D)}\sup_{0\neq v\in H_0^1(\D)} \frac{\abs{a(u,v)}}{\He{u}\He{v}}\le \frac{c_2}k,
 \end{align}
for some positive constants $c_1<c_2$ independent of $k$, where $\He{\cdot}$ is some energy norm defined in \eqref{eq:enorm}. Note that the above inf--sup condition on $H_0^1(\D)$ is not a direct consequence of the inf--sup condition of the original Helmholtz problem \cite{chandler2008wave} and the convergence estimates of the truncated PML problem, since they are valid only on $H^1(\Om)$. Such an inf--sup condition in \eqref{infsup} is useful in the convergence analysis of the truncated PML problem and the analysis of the source transfer domain decomposition method for the truncated PML problem \cite{chen2013source}. In order to carry out the preasymptotic analysis for the CIP-FEM, we need to derive the regularity estimate of the following adjoint problem to \eqref{PML}: Find $w\in H_0^1(\D)$ such that
\begin{align}\label{w}
a(v,w) = \inD{v}{\hu-u_h} \quad\forall v\in H_0^1(\D),
\end{align}
where $u_h$ is the CIP-FE solution. Since the adjoint problem and  the original problem \eqref{PML} are quite similar, so are theirs analysis.    
For easy of presentation, we analyze the original problem  instead. In precise, we derive the following stability estimates for the truncated PML problem \eqref{PML}:
\begin{equation}\label{eq:stab}
k^{-1}\sHt{\hu}{\OchO}+\He{\hu}+k\Lt{\hu}{\D} \leq C\Lt{f}{\D}.
\end{equation}
Since, usually, $\supp(\hu-u_h)\not\subset\Om$, the $f$ in the above estimates is allowed be nonzero in the PML region $\hat\Om$, while in the other estimates regarding  the truncated PML problem  \eqref{PML}, such as the convergence estimates and the error estimates of its CIPFE approximation, it is still assumed that $\supp f\subset\Om$.
Clearly,  the nonzero source in $\hat\Om$ brings about the backward waves, which is one reason that the proof of \eqref{eq:stab} is nontrivial. 
We remark that the authors found that it is not easy to extend the approach in \cite{melenk2010convergence} using Fourier transforms and the analysis in \cite{melenk95a,cf06} using the Rellich identity to the truncated PML problem with complex variable coefficients.  
Our key idea for proving \eqref{eq:stab} is to use the harmonic expansion of the truncated PML solution and analyze each term carefully in the expansion by using various properties of the Bessel functions (cf. \cite{bao2005,Boubendir2013Wave,mswy2018}). The estimates of the inf--sup constant in \eqref{infsup} are proved by using \eqref{eq:stab} and following the proofs in \cite{chandler2008wave}. Furthermore, we derive preasymptotic error estimates for the CIP-FEM by using the regularity estimate of $w$ and the modified duality argument developed in \cite{zhu2013preasymptotic}.

The outline of this paper is as follows. In  Section~\ref{sec:PML-Preliminary} we introduce the truncated PML problems in one, two, and three dimensions and derive the harmonic expansions of the truncated PML solutions.  Some preliminary results are also stated for further analysis. In  Section~\ref{sec:Anal-PML}, we derive the stability estimates, the inf--sup condition, and the convergence estimate with explicit dependence on the wave number $k$ for the truncated PML problem. In particular, the convergence rate is twice of the previous result \cite{chen2005}.  Section~\ref{sec:CIP-FEM} is devoted to the preasymptotic error estimates of CIP-FEM. In  Section~\ref{sec:Experimental}, some numerical tests are provided to verify the preasymptotic error estimates and to show that the penalty parameter in the CIP-FEM may be tuned to reduce greatly the pollution errors. 
	
Throughout the paper, $C$ is used to denote a generic positive constant which is independent of $h,k,f$, and the penalty parameters. We also use the shorthand notation $A\ls B$ and $A\gtrsim B$ for the inequality $A\leq CB$ and $A\geq CB$. $A\eqsim B$ is a notation for the statement $A\ls B$ and $A\gtrsim B$. In addition, the standard space, norm, and inner product notation are adopted. Their definitions can be found in \cite{brenner2007mathematical,ciarlet2002finite}. In particular, $(\cdot,\cdot)_Q$ and $\ine{\cdot}{\cdot}$ denote the $L^2$-inner product on complex-valued $L^2(Q)$ and $L^2(e)$ spaces, respectively. For simplicity, it is assume that $R\eqsim 1$.

\section{PML and Preliminaries}\label{sec:PML-Preliminary}
In this section we introduce the truncated PML problem and some preliminary results for further analysis.
\subsection{The truncated PML problem}\label{ss:PML}
As discussed in \cite{cjm97,bramble2007analysis,collino1998perfectly,lassas1998existence}, the PML problem can be viewed as a complex coordinate stretching of the original scattering problem. We recall the PML obtained by stretching the  radial coordinate. Let 
\begin{align}\label{tr}
\tr:=\intr\al(s) ds=r\be(r) \text{ with } \al(r)=1+\i\si (r), ~\be(r)=1+\i \de(r),
\end{align}
where
\begin{equation} \label{def:medium}
\si (r)=\left\{
\begin{aligned}
& 0, & 0\leq r \leq R, \\
& \siz,  & r>R,
\end{aligned}
\right. \qquad \de(r)=\left\{
\begin{aligned}
& 0, & 0\leq r \leq R, \\
& \frac{\siz(r-R)}{r},  & r>R, 
\end{aligned}
\right.
\end{equation}
and $\siz>0$ is a constant. Note that we have assumed that the PML medium property $\si$ to be constant to simplify the analysis, while our ideas also apply to variable PML medium property (see Remark~\ref{rm:varPML}(iii)). The PML equation is obtained from the Helmholtz equation \eqref{eq:Helm} by replacing the radial coordinate $r$ by $\tr$. For example, in the case of two dimensions ($d=2$), 
 the Helmholtz equation \eqref{eq:Helm} may be rewritten in polar coordinates  as follows. 
\begin{equation}\label{eq:Helm:polar}
-\frac{1}{r}\frac{\partial}{\partial r} \left(r\frac{\partial u}{\partial r}\right) - \frac{1}{r^2}\frac{\partial^2 u}{\partial \theta^2} - k^2 u = f.		
\end{equation}
Then the PML equation is given by
\begin{equation*}
-\frac{1}{\tr}\frac{\partial}{\partial \tr} \left(\tr\frac{\partial \tu}{\partial \tr}\right) - \frac{1}{\tr^2}\frac{\partial^2 \tu}{\partial \theta^2} - k^2 \tu = f,		
\end{equation*}
where $\tu(r,\theta):=u(\tr,\theta)$. Noting that $\frac{\pa}{\pa r}=\al(r)\frac{\pa}{\pa \tr}$, the above equation is rewritten as:
\begin{equation}\label{eq:Helm:trans}
-\frac{1}{r}\frac{\partial}{\partial r} \left(\frac{\be r}{\al}\frac{\partial \tu}{\partial r}\right) - \frac{\al}{\be r^2}\frac{\partial^2 \tu}{\partial \theta^2} - \al\be k^2 \tu = f.
\end{equation}
We note that $\tu = u$ in $\Om$ and $\tu$ decays exponentially away from the boundary of $\Om$ (see \cite[etc.]{chen2005,bao2005}). Therefore, in practice, the PML problem is truncated  at $r=\hR$ for some $\hR>R$ where $\tu$ is sufficiently small. Denote by $\hOm = \{ x\in \R^d:\abs{x}\in (R,\hR) \}$ and by $\D = \mathcal{B}(\hR)=\Om\cup\Gamma\cup\hOm,~\hat{\Gamma} := \partial \D$. Let $L:=\hR-R$ denotes the thickness of PML. Then we arrive at the following truncated PML problem:
\begin{equation}\label{eq:PML:2d}
\left\{
\begin{aligned}
-\frac{1}{r}\frac{\partial}{\partial r} \left(\frac{\be r}{\al}\frac{\partial \hu}{\partial r}\right) - \frac{\al}{\be r^2}\frac{\partial^2 \hu}{\partial \theta^2} - \al\be k^2 \hu = f \quad  &\mbox{in } \D, \\
\hu =0 \quad &\mbox{on } \hat{\Gamma}. 
\end{aligned}
\right. \quad (d=2)
\end{equation}
The truncated PML equations for one and three dimensional cases may be derived in a similar way (see \cite[etc.]{bramble2007analysis}): 
\begin{equation}\label{eq:PML:1d}
\left\{
\begin{aligned}
-\frac{d}{dx}\left(\frac{1}{\al}\frac{d\hu}{dx}\right)-\al k^2\hu =f \quad  &\mbox{in } \D, \\
\hu =0 \quad &\mbox{on } \hat{\Gamma}. 
\end{aligned}
\right. \quad (d=1)
\end{equation}
\begin{equation}\label{eq:PML:3d}
\left\{
\begin{aligned}
-\frac{1}{r^2} \frac{\partial}{\partial r} \left(\frac{\be^2 r^2}{\al}\frac{\partial \hu}{\partial r}\right) - \frac{\al}{r^2}\Delta_S \hu - \al\be^2 k^2 \hu = f \quad &\mbox{in  } \D, \\
\hu=0 \quad & \mbox{on } \hat{\Gamma}.
\end{aligned}
\right. \quad (d=3)
\end{equation}
where $\Delta_S = \frac{1}{\sin\theta}\frac{\partial}{\partial\theta} \left(\sin\theta\frac{\partial}{\partial\theta}\right) + \frac{1}{\sin^2\theta}\frac{\partial^2}{\partial\vp ^2}$ is the Laplace-Beltrami operator on $S$. 

In Cartesian coordinates, the truncated PML problems \eqref{eq:PML:2d}--\eqref{eq:PML:3d} can be rewritten in the following unified form:
\begin{align}
-\na\cdot(A\na \hu)-Bk^2\hu = f \qquad & \mbox{in } \D \label{eq:PML:all}\\
\hu = 0 \qquad & \mbox{on } \hat{\Gamma}\label{eq:PML:bound:all}
\end{align}
where
\begin{align*}
A=HDH^T,\quad B=\al(r)\be^{d-1}(r),
\end{align*}
and
\[
\begin{array}{cc}
D=\frac{1}{\al(x)}, ~ H=1, & \text{for } d=1,\\
D=\begin{pmatrix}
	\frac{\be(r)}{\al(r)} & 0 \\
	0 & \frac{\al(r)}{\be(r)}	
  \end{pmatrix},~
H=\begin{pmatrix}
	\cos\theta & -\sin\theta \\
	\sin\theta & \cos\theta
  \end{pmatrix}, & \text{for } d=2,\\
D=\begin{pmatrix}
	\frac{\be^2(r)}{\al(r)} & 0      & 0    \\
	0                       & \al(r) & 0    \\
	0                       & 0      & \al(r)
  \end{pmatrix},~
H=\begin{pmatrix}
	\sin\theta\cos\vp & \cos\theta\cos\vp & -\sin\vp \\
	\sin\theta\sin\vp & \cos\theta\sin\vp & \cos\vp \\
	\cos\theta            & -\sin\theta           & 0 
  \end{pmatrix}, & \text{for } d=3.
\end{array}
\]
The variational formulation of the truncated PML problem \eqref{eq:PML:all}--\eqref{eq:PML:bound:all} reads as: Find $\hu\in H_0^1(\D)$ such that
\begin{equation}\label{eq:PML:Var:1}
 a(\hu,v) = \inD{f}{v} \quad\forall v\in H_0^1(\D)
\end{equation}
where
\begin{align}\label{a}
a(u,v) &:= \inD{A\na u}{\na v}-k^2\inD{Bu}{v}.
\end{align}
Define the energy norm $\He{\cdot}$ by
\begin{equation}\label{eq:enorm}
\He{v} := \left(\Re(a(v,v))+2k^2\Lt{v}{\D}^2\right)^{\frac 12}.
\end{equation}
The definition is reasonable since it can be shown that $\Re(a(v,v))+2k^2\Lt{v}{\D}^2 > 0$ for any $0\neq v\in H^1(\D)$. For example, for the 2D case, this is a consequence of the following formula of $\Re(a(v,v))$ and the fact that $0\leq\delta\leq\sigma$.
\begin{align*}
\Re a(v,v) &= \int_0^{2\pi}\hskip -7pt\int_0^{\hat R} \Big(\frac{1+\sigma\delta}{1+\sigma ^2}r\abs{v_r}^2 + \frac{1+\sigma\delta}{1+\delta ^2}\frac{1}{r}\abs{v_\theta}^2 + (\sigma\delta-1)k^2 r\abs{v}^2\Big)dr d\theta .
\end{align*}

\subsection{Harmonic expansions} In this subsection, we write the solutions to the original scattering problem \eqref{eq:Helm}--\eqref{eq:Somm} and the truncated PML problem \eqref{eq:PML:all}--\eqref{eq:PML:bound:all} into harmonic expansions.

\subsubsection{1D case} A simple calculation shows that the solution $u$ to the Helmholtz problem \eqref{eq:Helm}--\eqref{eq:Somm} in $\R^1$ is given by
\begin{equation}\label{eq:un:1d}
u(x) = -\frac{1}{2\i k}e^{-\i kx}\int_x^{+\infty} e^{\i kt}f(t)dt - \frac{1}{2\i k}e^{\i kx}\int_{-\infty}^x e^{-\i kt}f(t)dt
\end{equation}
Consequently, the PML solution $\hu$ of \eqref{eq:PML:1d} is given by
\begin{equation}
\begin{aligned}\label{eq:hun:1d}
\hu(x) = &-\frac{1}{2\i k}e^{-\i k\tx}\int_x^{\hR} e^{\i k\tt}f(t)dt - \frac{1}{2\i k}e^{\i k\tx}\int_{-\hR}^x e^{-\i k\tt}f(t)dt \\
 & +D_1 e^{-\i k\tx}+D_2 e^{\i k\tx} 
\end{aligned}
\end{equation}
where $\tx:=-\widetilde{(-x)}$ for $x<0$ and
\begin{equation}\label{eq:hun:1d:coe}
D_1 = \frac{1}{2\i k}\frac{e^{2\i k\thR}J_2-J_1}{e^{2\i k\thR}-e^{-2\i k\thR}},\qquad 
D_2 = \frac{1}{2\i k}\frac{e^{2\i k\thR}J_1-J_2}{e^{2\i k\thR}-e^{-2\i k\thR}}
\end{equation}
with $J_1 := \int_{-\hR}^{\hR} e^{-\i k\tt}f(t)dt$ and $J_2 := \int_{-\hR}^{\hR} e^{\i k\tt}f(t)dt$. 

\subsubsection{2D case}
We solve the problems by separation of variables. Recall that every function $w\in L^2(\R^2)$ has the following Fourier expansion:
\begin{equation}\label{eq:Harm-expansions}
w(r,\theta) = \sum_{n\in\Z} w_n(r) e^{\i n\theta}, \quad w_n(r) =\frac{1}{2\pi} \int_0^{2\pi} w(r,\theta)e^{-\i n\theta} d\theta.
\end{equation}
Note that we have
\[
\Lt{w}{\R^2}^2=2\pi\sum_{n\in\Z}\int_0^{+\infty} r\abs{w_n(r)}^2 dr.
\]
By substituting the Fourier expansion of $u$ into  the PDE \eqref{eq:Helm:polar} we obtain the following ODE system of $u_n$:
\begin{equation}\label{eq:Helm:ODE}
-\frac{1}{r}\frac{d}{dr}\left(r\frac{du_n}{dr}\right)-\left(k^2-\frac{n^2}{r^2}\right)u_n = f_n. 
\end{equation}
Introduce the variable $s=kr$ and rewrite the above equations as:
\begin{equation*}
\frac{d^2 u_n}{ds^2}+\frac{1}{s}\frac{du_n}{ds}+\left(1-\frac{n^2}{s^2}\right)u_n = -\frac{1}{k^2}f_n. 
\end{equation*}
Applying the general ODE theory \cite{coddington1955theory} and the definitions of the Bessel functions (see, \cite[10.2.1]{olver2010nist}), and noting the Wronskian $\mathscr{W}\{\Jv(z),\Hv(z)\}=\frac{2\i}{\pi z}$ (see, \cite[10.5.3]{olver2010nist}), it can be shown that
\begin{equation}\label{eq:un}
u_n(r) = \frac{\pi\i}{2}\J (kr)\int_r^{+\infty} \H(kt)f_n(t)t dt + \frac{\pi\i}{2}\H (kr)\intr \J(kt)f_n(t)t dt,
\end{equation}
where $\Jv(z)$ and $\Hv(z)$ denote the usual Bessel and Hankel functions of the first kind of order $\nu$. Analogously, for the solution $\hu$ to the PML problem \eqref{eq:PML:2d}, we have
\begin{equation}\label{eq:PML:ODE:2d}
-\frac{1}{r}\frac{d}{dr}\left(\frac{\be r}{\al}\frac{d\hu_n}{dr}\right)-\left(\al\be k^2-\frac{\al n^2}{\be r^2}\right)\hu_n = f_n. 
\end{equation}
From the boundary condition \eqref{eq:PML:bound:all}, there holds
\begin{equation}\label{eq:PML:ODE:bound}
\hu_n(\hR)=0,\quad n\in\Z.
\end{equation}
Furthermore, using  \eqref{eq:Harm-expansions} we have
\begin{equation}\label{eq:PML:ODE:bound:1}
\hu_n(0)=\frac{1}{2\pi}\int_0^{2\pi}\hu(0)e^{-\i n\theta} d\theta 
=0 \text{ for } 0\neq n\in\Z.
\end{equation}
Since $-\Delta\hu_0-k^2\hu_0=f_0$ in the ball $\B(r)$ for $r<R$, we have
$\big|\frac{1}{2\pi r}\int_{\pa\B(r)}\frac{\pa\hu_0}{\pa r}\big|=\frac{1}{2\pi r}\big|\int_{\B(r)}\big(f_0+k^2\hu_0\big)\big|\ls\Lt{f_0}{\B(r)}+k^2\Lt{\hu_0}{\B(r)},$
which implies that 
\begin{equation}\label{eq:PML:ODE:bound:2}
\frac{d\hu_0}{dr}(0)=0.
\end{equation}
Solving these ODEs \eqref{eq:PML:ODE:2d}--\eqref{eq:PML:ODE:bound:2} leads to
\begin{equation}\label{eq:hun}
\begin{aligned}
\hu_n(r) = & {}\frac{\pi\i}{2}\J (k\tr)\int_r^{\hR} \H(k\tt)f_n(t)t dt + \frac{\pi\i}{2}\H (k\tr)\intr \J(k\tt)f_n(t)t dt \\
 & {}+ \hC_n \J(k\tr). 
\end{aligned}
\end{equation}
where
\begin{equation}\label{eq:hC}
\hC_n = -\frac{\pi\i}{2}\frac{\H(k\thR)}{\J(k\thR)}\int_0^{\hR}\J(k\tt)f_n(t)tdt.
\end{equation}
Note that from \cite[10.21(i)]{olver2010nist}, the zeros of $\Jv(z)$ are all real for $\nu\geq -1$. Since $\thR$ is not real and $J_{-n}=(-1)^n\J$ for integer $n$, we have $\J(k\thR)\neq 0$, $n\in\Z$.

\subsubsection{3D case} 
Let $Y_{\l}^{m}(\theta,\vp)$ be the standard spherical harmonics, which form an orthonormal basis of the square-integrable functions on the unit sphere and satisfy (see, \cite{nedelec2001acoustic}):
\begin{equation*}
\Delta_S Y_{\l}^{m} +\l(\l+1)Y_{\l}^{m} = 0,\quad \l\in\N,~m\in\Z\cap [-\l,\l].
\end{equation*}
The solutions $u$ and $\hu$ satisfy the following harmonic expansions:
\begin{equation}\label{eq:harm:3d}
u = \sum_{\l=0}^{+\infty} \sum_{m=-\l}^{\l} u_\l^m(r) Y_{\l}^{m}(\theta,\vp),\quad \hu = \sum_{\l=0}^{+\infty} \sum_{m=-\l}^{\l} \hu_\l^m(r) Y_{\l}^{m}(\theta,\vp),
\end{equation}
where $(r,\theta,\vp)$ are the spherical coordinates. Similarly, we have
\[
\Lt{u}{\R^3}^2=\sum_{\l=0}^{+\infty}\sum_{m=-\l}^{\l}\int_0^{+\infty} r^2\abs{u_\l^m(r)}^2 dr
\]
The coefficients $u_\l^m$ and $\hu_\l^m$ satisfy
\begin{equation}\label{eq:ODE:3d}
-\frac{1}{r^2}\frac{d}{dr}\left(r^2\frac{d u_\l^m}{dr}\right) - \left(k^2 -\frac{\l(\l+1)}{r^2}\right)u_\l^m = f_\l^m,
\end{equation}
\begin{equation}\label{eq:PML:ODE:3d}
-\frac{1}{r^2}\frac{d}{dr}\left(\frac{\be^2 r^2}{\al}\frac{d\hu_\l^m}{dr}\right) - \left(k^2\al\be^2 -\frac{\al\l(\l+1)}{r^2}\right)\hu_\l^m = f_\l^m.
\end{equation}
Similar to \eqref{eq:PML:ODE:bound}--\eqref{eq:PML:ODE:bound:2}, $\hu_\l^m$ satisfies the following boundary conditions:
\begin{equation}\label{eq:PML:ODE:bound:3D}
\hu_\l^m(\hR)=0,\quad \hu_\l^m(0)=0 \text{ for } 0\neq\l\in\N,
\quad\text{and } \lim_{r\to 0}r^{\frac12}\frac{d \hu_0^0}{d r}(r)=0.
\end{equation}
Noting that $v_\l^m:=r^\frac12 u_\l^m$ satisfies
\begin{align*}-\frac{1}{r}\frac{d}{dr}\left(r\frac{d v_\l^m}{dr}\right) - \left(k^2 -\frac{(\l+\frac12)^2}{r^2}\right)v_\l^m = r^\frac12 f_\l^m,
\end{align*}
it follows from \eqref{eq:Helm:ODE}--\eqref{eq:un} that
\begin{equation}\label{eq:un:3d}
\begin{aligned}
r^\frac12 u_\l^m(r) = &\frac{\pi\i}{2}J_{\l+\frac 12} (kr)\int_r^{+\infty} H_{\l+\frac 12}^{(1)}(kt)\big(t^\frac12 f_\l^m(t)\big)t dt \\
&+ \frac{\pi\i}{2}H_{\l+\frac 12}^{(1)} (kr)\intr J_{\l+\frac 12}(kt)\big(t^\frac12 f_\l^m(t)\big)t dt,
\end{aligned}
\end{equation}
Similarly, for the truncated PML solution, we have
\begin{equation}\label{eq:hulm}
\begin{aligned}
\tr^{\frac12}\hu_\l^m(r) = &~\frac{\pi\i}{2}J_{\l+\frac 12}(k\tr)\int_r^{\hR}H_{\l+\frac 12}^{(1)}(k\tt)\big(\be^{-\frac12}t^\frac12 f_\l^m(t)\big)tdt \\
 &~+ \frac{\pi\i}{2}H_{\l+\frac 12}^{(1)}(k\tr)\intr J_{\l+\frac 12}(k\tt)\big(\be^{-\frac12}t^\frac12 f_\l^m(t)\big)tdt + \hC_{\l+\frac 12}J_{\l+\frac 12}(k\tr),
\end{aligned}
\end{equation}
where
\begin{equation}\label{eq:hC2}
\hC_{\l+\frac 12} = -\frac{\pi\i}{2}\frac{H_{\l+\frac 12}^{(1)}(k\thR)}{J_{\l+\frac 12}(k\thR)}\int_0^{\hR}J_{\l+\frac 12}(k\tt)\big(\be^{-\frac12}t^\frac12 f_\l^m(t)\big)tdt.
\end{equation}
Similarly to \eqref{eq:hC}, $J_{\l+\frac 12}(k\thR)\neq 0$ for $\l\in\N$.

\subsection{Stability estimates of the Helmholtz problem}
The following lemma is proved in \cite[Lemma 3.5]{melenk2010convergence}.
\begin{lemma}\label{lem:stab:Helm-Somm}
For the solution $u$ to the problem \eqref{eq:Helm}--\eqref{eq:Somm}, there holds
\begin{equation*}
k^{-1}\Ht{u}{\Om} +\Ho{u}{\Om} +k\Lt{u}{\Om} \leq C\Lt{f}{\Om}.
\end{equation*}
\end{lemma}

\begin{remark}
(i) In \cite{melenk2010convergence}, the stability estimates above are proved for the solution $u^{\rm DtN}$ to the Helmholtz problem \eqref{eq:Helm} with the transparent DtN boundary condition at $\Gamma$. Apparently, $u^{\rm DtN}|_\Om=u|_\Om$.

(ii) From \eqref{eq:Helm:ODE}, it clear that $u=u_n(r)e^{\i n\theta}$ is the solution to the problem \eqref{eq:Helm}--\eqref{eq:Somm} with $f=f_n(r)e^{\i n\theta}$. Therefore, 
as consequences of the above lemma we have:
\begin{align}
\intR r\abs{u_n(r)}^2dr &\ls{} \frac{1}{k^2}\intR r\abs{f_n(r)}^2dr,\quad n\in\Z \quad(d=2). \label{eq:stab:un:2d}
\end{align}
Similarly, there holds
\begin{align}
\intR r^2\abs{u_\l^m(r)}^2dr &\ls{} \frac{1}{k^2}\intR r^2\abs{f_\l^m(r)}^2dr,\; \l\in\N,~m\in\Z\cap[-\l,\l]  \quad(d=3).\label{eq:stab:un:3d}
\end{align}
\end{remark}


\subsection{Properties of the Special functions}
In this subsection, we state some properties of the Bessel, Hankel functions, and the Modified Bessel functions, which are required for the stability estimates of truncated PML problem.
\begin{lemma}\label{lem:prop:Bessel-Hankel}
For any $\nu\in\R, z\in\C_{++}=\left\{z\in \C:\Im(z)\geq 0, \Re(z)\geq 0\right\}$ and $x\in \R$ such that $0<x\leq\abs{z}$, we have ,
\begin{equation}\label{eq:Hankel:bound}
\abs{H_{\nu}^{(1)}(z)} \leq e^{-\Im(z)\left(1-\frac{x^2}{\abs{z}^2}\right)^{1/2}}\abs{H_{\nu}^{(1)}(x)},
\end{equation}
In addition, for any $n\in\Z,~\l\in\N,~z\in\C$, there holds 
\begin{equation}\label{eq:Bessel:bound}
\abs{\J(z)}\leq e^{\abs{\Im(z)}},\quad \abs{J_{\l+\frac12}(z)}\leq \abs{\frac{2z}{\pi}}^{\frac12}e^{\abs{\Im(z)}}.
\end{equation}
\end{lemma}
\begin{proof}
We refer to \cite[Lemma 2.2]{chen2005} for \eqref{eq:Hankel:bound} and \cite[10.14.3]{olver2010nist} for the first inequality in \eqref{eq:Bessel:bound}. On the other hand, it follows from \cite[10.47.3,10.54.2]{olver2010nist} that
\begin{equation*}
\abs{J_{\l+\frac12}(z)}\leq \abs{\frac{2z}{\pi}}^{\frac12}\frac{e^{\abs{\Im(z)}}}{2}\abs{\int_{-1}^1 \abs{P_\l(t)}dt}
\end{equation*}
where $P_\l(t)$ is the Legendre function, which implies the second inequality in \eqref{eq:Bessel:bound}.
\end{proof}

In the following lemma which is proved in \cite[Lemma 5.1 and A.1]{bao2005}, we introduce the uniform asymptotic expansions of the Modified Bessel functions $I_\nu(z)$ and $K_\nu(z)$. To estimate $\Jv(z)$ and $\Hv(z)$, we use the following relations \cite[10.27.6--8]{olver2010nist}:
\begin{align}\label{eq:Bessel:relation}
\Jv(z)&=e^{\frac{\nu\pi\i}{2}}I_{\nu}(-\i z),~\Hv(z)=-\frac{2\i}{\pi}e^{-\frac{\nu\pi\i}{2}}K_{\nu}(-\i z), \quad -\frac{\pi}{2}\leq \arg z\leq \pi.
\end{align}
\begin{lemma}\label{lem:prop:Modi-Bessel}
Denote by $z=x+\i y=re^{\i\theta}\in\C,x,y\in\R.$ Assume that $z$ satisfies $\abs{\arg z}\leq \frac{\pi}{4}$. Then
\begin{align}
I_{\nu}(\nu z) = & {}\left(\frac{1}{2\pi\nu}\right)^{1/2}\frac{e^{\nu\xi}}{(1+z^2)^{1/4}}\frac{1+\eta_1(\nu,z)}{1+\eta_1(\nu,\infty)},\label{eq:Modi-Bessel:1}\\
K_{\nu}(\nu z) = & {}\left(\frac{\pi}{2\nu}\right)^{1/2}\frac{e^{-\nu\xi}}{(1+z^2)^{1/4}}\big(1+\eta_2(\nu,z)\big). \label{eq:Modi-Bessel:2}
\end{align}
where $\xi=\left(1+z^2\right)^{1/2}+\ln \frac{z}{1+\left(1+z^2\right)^{1/2}}$. Moreover, the error terms $\eta_j,~j=1,2$ are bounded by
\begin{equation}\label{eq:Modi-Bessel:err:eta}
\abs{\eta_j(\nu,z)} \leq \hat{M}(\nu,z):= \exp\left(\frac{2M(z)}{\nu}\right) \frac{2M(z)}{\nu} 
\end{equation}
where $M(z)$ is defined as follows:
\begin{equation}\label{eq:Modi-Bessel:err:M}
M(z)=\frac{1}{12}+\frac 1{6\sqrt5}+\frac{\abs{\Im(z)}}{\Re(z)} \min\left\{ \left(\frac 4{27}\right)^{1/4},\frac{1}{\Re(z)} \right\}. 
\end{equation}
In addition, the following estimates for $\xi$ hold:
\begin{enumerate}[\rm (i)]
\item $\Re(\xi) \geq \xi(x)> x-\frac{1}{2x}$, for $x\geq \frac{7}{10}$, \label{eq:Modi-Bessel:err:1}
\item $\Re(\xi)$ is increasing in $r$ and decreasing in $\abs{\theta}$. \label{eq:Modi-Bessel:err:3}
\end{enumerate}
\end{lemma}

The following lemma gives more properties for the Hankel functions.
\begin{lemma}\label{lem:prop:Bessel:zero}
For any $z\in\C$, the following formulas hold for $j=1,2$ 
\begin{align}
H_0^{(j)}(z) & = {}\left(\frac{2}{\pi z}\right)^{\frac 12}e^{(-1)^{j-1}\i w}\left(1+R_j(z)\right),\label{eq:Bessel:zero}\\
H_{\frac 12}^{(j)}(z) & = {}(-1)^{j}\i\left(\frac{2}{\pi z}\right)^{\frac12}e^{(-1)^{j-1}\i z}.\label{eq:Modi-Bessel:zero}
\end{align}
where $w=z-\frac {\pi}{4}$. Furthermore, if $\abs{z}\geq 1$ and $\abs{\arg{z}}\leq\frac{\pi}{2}$, then  
\begin{align}
\abs{R_j(z)}\leq \frac{\pi}{8}e^{\frac{\pi}{8}}<1.\label{eq:Bessel:zero:bound}
\end{align}
\end{lemma}
\begin{proof}
We refer to \cite[10.47.5--6, 10.49.6--7]{olver2010nist} for \eqref{eq:Modi-Bessel:zero}.   \eqref{eq:Bessel:zero} and \eqref{eq:Bessel:zero:bound} follow from \cite[10.17.13--15]{olver2010nist} and some simple calculations.
\end{proof}

\section{Analysis of the truncated PML problem}\label{sec:Anal-PML}
The analysis of the truncated PML problem is split into four parts. The first part is the $L^2$-stability estimates for \eqref{eq:PML:all}--\eqref{eq:PML:bound:all}, which are quite significant in our work. Then, by applying the variational formula \eqref{eq:PML:Var:1}, we obtain the $H^1$-stability estimates and show that the inf--sup condition constant $\mu \eqsim k^{-1}$. Furthermore, we derive the convergence estimate for the truncated PML problem. Finally, the $H^2$-estimates are derived as an immediate consequence.

\subsection{$L^2$-stability estimates}
In this subsection we prove the following theorem.
\begin{theorem}\label{thm:stab:L2} 
Suppose $R\eqsim\hR\eqsim 1$. Assume that $k\siz L\geq 1$ for 1D case and that 
  
\begin{equation}\label{eq:assumption}
kR\geq 1 \quad\text{ and }\quad k\siz L\geq\max\Big\{ 2kR+\sqrt{3}kL,10 \Big\} \quad\text{ for 2D and 3D cases.}
\end{equation}
Then the truncated PML problem \eqref{eq:PML:all}--\eqref{eq:PML:bound:all} satisfies the following estimate:
\begin{equation}\label{eq:stab:L2} 
\Lt{\hu}{\D} \ls \frac{1}{k}\Lt{f}{\D}.
\end{equation}
\end{theorem}
\begin{remark}
(i) Clearly the assumption \eqref{eq:assumption} is not strict for high wave number problems.

(ii) The assumption of $\hR \eqsim 1$ is merely for the ease of presentation. It can be removed but the formulations of the results would be more complicated.

(iii) It is possible to  extend the results of this paper to variable PML medium properties, in particular, the following parameters considered in \cite{chen2005}, 
\begin{equation*} 
\si (r)=\left\{
\begin{aligned}
& 0, & 0\leq r \leq R, \\
& \siz\left(\frac{r-R}{\hR-R}\right)^m,  & r>R,
\end{aligned}
\right.
\end{equation*}
where  $m\in \N$ and $\siz >0$ are constants. For example, for the 1D case, some simple calculations show that  \eqref{eq:stab:L2} holds if $\frac{k\siz L}{m+1}\geq 1$. For the 2D and 3D cases,   by combining the ideas from this paper and  \cite{chen2005,cf06}, the same stability estimate may be proved under some appropriate modifications of the conditions of this paper.  The results on variable PML medium properties will be reported in a future work. 

(iv) Since the proof of Theorem \ref{thm:stab:L2} for 1D case is simple by applying \eqref{eq:hun:1d}, we omit it to save space.

\label{rm:varPML}
\end{remark}

\subsubsection{Proof of Theorem \ref{thm:stab:L2} for 2D case} 
In order to analyze $\hu_n$ defined by the equations \eqref{eq:PML:ODE:2d}--\eqref{eq:PML:ODE:bound:2}, we introduce the following sesquilinear form
\begin{equation}\label{eq:stab:4}
a_\nu(w,v):=\inthR\left( \frac{\be}{\al} r \frac{dw}{dr}\frac{d\bar{v}}{dr}+\Big(\frac{\nu^2}{r}\frac{\al}{\be}-k^2\al\be r\Big)w\bar{v} \right) dr .
\end{equation}
The following lemma gives some coercivity properties of $a_\nu$.
\begin{lemma}\label{lem:stab:an}
Assume that $\siz\geq\sqrt{3}$. For any $v\in H^1((0,\hR))$ and $v(\hR)=0$, there exists a constant $\lam= \frac{R^{3/2}}{3\siz\hR^2L^{1/2}}$ such that
\begin{enumerate}
\item[\rm (i)] If $\nu^2\geq k^2R^2+\lam k\hR^2$, then
\end{enumerate}
\begin{equation}\label{eq:stab:5}
\Re\big(a_\nu(v,v)\big) \gtrsim \inthR\left(\frac{1}{1+\siz^2}r\abs{v_r}^2+\lam kr\abs{v}^2\right) dr
\end{equation}
\begin{enumerate}
\item[\rm (ii)] If $\nu^2 < k^2R^2+\lam k\hR^2$, then
\end{enumerate}
\begin{equation}\label{eq:stab:6}
-\Im\big(a_\nu(v,v)\big) \gtrsim \intRhR\left(\frac{\siz R}{(1+\siz^2)\hR}r\abs{v_r}^2+\lam kr\abs{v}^2\right) dr.
\end{equation}
\end{lemma}
\begin{proof}
From \eqref{eq:stab:4} and \eqref{def:medium}, we get
\begin{align}
\Re\big(a_\nu(v,v)\big) & = {} \inthR \left(\frac{1+\si \de}{1+\si ^2}r\abs{v_r}^2 + \Big(\frac{1+\si \de}{1+\de^2}\frac{\nu^2}{r}-k^2(1-\si \de)r\Big)\abs{v}^2 \right) dr \label{eq:stab:7} \\
-\Im\big(a_\nu(v,v)\big) & = {} \intRhR \left(\frac{\si -\de}{1+\si ^2}r\abs{v_r}^2 + \Big(\frac{\de-\si }{1+\de^2}\frac{\nu^2}{r}+k^2(\si +\de)r\Big)\abs{v}^2 \right) dr \label{eq:stab:8}
\end{align}

We first prove (i). From $0\leq\de\leq\sigma$ and \eqref{eq:stab:7}, it suffices to prove the following inequality
\begin{equation*}
\frac{1+\si\de}{1+\de^2}\frac{\nu^2}{r}-k^2(1-\si\de)r \geq \lam kr.
\end{equation*}
 The following inequality is a sufficient condition of the above one:
\begin{equation}\label{eq:stab:8:1}
\nu^2 \geq \frac{1+\de^2}{1+\si\de}(1-\si\de)k^2 r^2+\lam kr^2 =: g(r) + \lam kr^2
\end{equation}
Since $\siz^2\geq 3$, it can be shown that $R$ is the maximum point of $g(r)$ by verifying its monotonicity. That is, 
\[g(r)\leq g(R)=k^2R^2,\quad r\in (0,\hR)\]
and hence, \eqref{eq:stab:8:1} holds.

Next we prove (ii). From \eqref{eq:stab:8} and \eqref{def:medium}, we get
\begin{equation*}
\begin{aligned}
-\Im\big(a_\nu(v,v)\big) & \geq {} \intRhR \left(\frac{\si R}{1+\si^2}\abs{v_r}^2 + \Big(-\si\frac{k^2R^2+\lam k\hR^2}{r}+k^2(\si+\de)r\Big)\abs{v}^2 \right) dr \\
 & \geq {} \intRhR \left(\frac{\si R}{1+\si^2}\abs{v_r}^2 + \Big(k^2\de r - \frac{\si\lam k\hR^2}{R}\Big)\abs{v}^2 \right) dr
\end{aligned}
\end{equation*}
Note that $v(\hR)=0$, there holds
\[
\abs{v(r)}^2=\abs{\intrhR v_r(s) ds}^2 \leq\left(\intRhR\abs{v_r}dr\right)^2\leq L\intRhR\abs{v_r}^2dr,\quad R\leq r\leq\hR, 
\]
hence,
\begin{equation}\label{eq:stab:10}
\frac{k}{\tau L}\int_R^{R+\tau /k} \abs{v}^2 dr\leq\intRhR\abs{v_r}^2 dr.
\end{equation}
where $\tau$ should be chosen as a positive constant independent of $k$. Then
\begin{equation*}
\begin{aligned}
-\Im\big(a_\nu(v,v)\big) \geq &{} \frac{\siz R}{2(1+\siz^2)}\intRhR \abs{v_r}^2 dr + \frac{\siz R}{2(1+\siz^2)}\frac{k}{\tau L}\int_R^{R+\tau /k} \abs{v}^2 dr \\
 &{} - \int_R^{R+\tau /k}\frac{\si\lam k\hR^2}{R}\abs{v}^2 dr +\int_{R+\tau /k}^{\hR} \Big(k^2\de r - \frac{\si\lam k\hR^2}{R}\Big)\abs{v}^2 dr \\
\geq &{} \frac{\siz R}{2(1+\siz^2)}\intRhR \abs{v_r}^2 dr + \int_R^{R+\tau /k}k\siz\Big( \frac{R}{2\tau L(1+\siz^2)}-\lam\frac{\hR^2}{R}\Big) \abs{v}^2 dr \\
 &{} + \int_{R+\tau /k}^{\hR} k\siz\Big(\tau - \lam\frac{\hR^2}{R}\Big)\abs{v}^2 dr,
\end{aligned}
\end{equation*}
we take $\tau = \frac{R}{2\tau L(1+\siz^2)}$, i.e., $\tau= \frac{R^{1/2}}{\sqrt{2(1+\siz^2)}L^{1/2}}$. To get \eqref{eq:stab:6}, it remains to prove that
\begin{equation}\label{eq:stab:11}
k\siz\Big(\tau - \lam\frac{\hR^2}{R}\Big) \geq \lam k\hR ,
\end{equation}
Note that $\siz\geq \sqrt{3}$, we choose $\lam= \frac{R^{3/2}}{3\siz\hR^2L^{1/2}}$. Some simple calculations lead to
\[
\lam \leq \frac{R^{3/2}\siz}{2\sqrt{2(1+\siz^2)}\hR^2 L^{1/2}} \leq \frac{\tau}{\frac{\hR}{\siz}+\frac{\hR^2}{R}}
\]
which implies \eqref{eq:stab:11}, and hence, \eqref{eq:stab:6} holds. This completes the proof of the lemma. 
\end{proof}

Clearly, \eqref{eq:assumption} is a sufficient condition of the following set of conditions
\begin{equation}\label{assumption:1}
\begin{aligned}
kR\geq 1, \siz \geq\sqrt{3}, \siz L\geq\hR, k\siz L\geq 10, L\geq \frac{2R}{\siz -1},\mbox{and } \siz^2L^2\geq R^2+\frac{\lam}{k}\hR^2
\end{aligned}
\end{equation}
but simpler. Here $\lam$ is defined as in  Lemma \ref{lem:stab:an}. 


Next we turn to estimate $\Lt{\hu}{\D}^2$. It suffices to prove that
\begin{equation}\label{eq:stab:12}
\inthR r\abs{\hu_n(r)}^2 dr \ls{} \frac{1}{k^2} \inthR r\abs{f_n(r)}^2 dr.
\end{equation}
We divide the proof of \eqref{eq:stab:12} into three cases.

\textbf{Case 1}: $n^2\geq k^2R^2+\lam k\hR^2$.  Note from \eqref{eq:PML:ODE:2d}-- \eqref{eq:PML:ODE:bound:1} and \eqref{eq:stab:4} that $\hu_n$ satisfies the following variational formulation:
\begin{equation}\label{eq:stab:3}
a_n(\hu_n,v) = \inthR rf_n\bar{v} dr,\quad \forall v\in H_0^1((0,\hR)).
\end{equation}
 It follows from \eqref{eq:stab:5} that
\begin{equation*}
\begin{aligned}
\inthR\lam kr\abs{\hu_n}^2 dr & \ls  \Re\big(a_n(\hu_n,\hu_n)\big) \ls \abs{\inthR r f_n \bar\hu_n dr} \\
 & \ls  \left(\frac{1}{\lam k}\inthR r\abs{f_n}^2dr\right)^{\frac 12} \left(\inthR \lam kr\abs{\hu_n}^2dr\right)^{\frac 12} ,
\end{aligned}
\end{equation*}
Thus,
\begin{equation}\label{eq:stab:14}
\inthR r\abs{\hu_n}^2 dr \ls  \frac{1}{\lam^2 k^2}\inthR r\abs{f_n}^2dr.
\end{equation}

\textbf{Case 2}: $1\leq n^2 < k^2R^2+\lam k\hR^2$. Similar to \eqref{eq:stab:14}, an application of \eqref{eq:stab:6} gives
\begin{equation}\label{eq:stab:15}
\intRhR r\abs{\hu_n}^2 dr \ls \frac{1}{\lam^2 k^2}\inthR r\abs{f_n}^2 dr + \intR r\abs{\hu_n}^2 dr.
\end{equation}

Next we estimate the last term in \eqref{eq:stab:15}. From \eqref{eq:hun} we have
\begin{equation}\label{eq:stab:18}
\hu_n(r) = p_n^{(1)}(r)+p_n^{(2)}(r)+p_n^{(3)}(r), \quad 0\leq r<R,
\end{equation}
where
\begin{align*}
& p_n^{(1)}(r) := \frac{\pi\i}{2}\J (kr)\intrR \H(kt)f_n(t)t dt + \frac{\pi\i}{2}\H (kr)\intr \J(kt)f_n(t)t dt \\
& p_n^{(2)}(r) := \frac{\pi\i}{2}\J(kr)\intRhR \H(k\tt)f_n(t)t dt \\
& p_n^{(3)}(r) := \hC_n\J(kr)
\end{align*}
Firstly, compared with \eqref{eq:un}, $p_n^{(1)} e^{\i n\theta}$ is the exact solution to the Helmholtz problem \eqref{eq:Helm}--\eqref{eq:Somm} with $f=f_n e^{\i n\theta}\chi_\Om$, here $\chi_\Om$ denotes the characteristic function of $\Om$. From \eqref{eq:stab:un:2d}, there holds
\begin{equation}\label{eq:stab:19}
\intR r\abs{p_n^{(1)}(r)}^2 dr \ls \frac {1}{k^2} \intR r\abs{f_n(r)}^2 dr.
\end{equation}
Secondly, it follows from \eqref{eq:Hankel:bound} that
\[
\abs{p_n^{(2)}(r)} \leq \abs{\frac{\pi\i}{2}\J(kr)\intRhR\abs{\H(kt)f_n(t)t}dt} = \Bigg\vert\underbrace{\frac{\pi\i}{2}\J(kr)\intRhR\H(kt)g_n(t)tdt}_{=:\psi(r)}\Bigg\vert
\]
where $g_n=f_n(t)\sgn\big(\H(kt)f_n(t)\big)$. Similarly, $\psi e^{\i n\theta}$ is the exact solution of \eqref{eq:Helm}--\eqref{eq:Somm} with $f=g_n e^{\i n\theta}\chi_{\hOm}$ and replacing $\Om$ by $\D$. Applying \eqref{eq:stab:un:2d}, we have
\begin{equation}\label{eq:stab:20}
\begin{aligned}
\intR r\abs{p_n^{(2)}(r)}^2 dr \leq \intR r\abs{\psi(r)}^2 dr & \ls \frac{1}{k^2}\intRhR r\abs{g_n(r)} ^2 dr \\
 & \ls \frac{1}{k^2}\intRhR r\abs{f_n(r)}^2 dr.
\end{aligned}
\end{equation}
It remains to estimate $\intR r\big|p_n^{(3)}(r)\big|^2 dr$,  especially to analyze the $\hC_n$ defined by \eqref{eq:hC}. Noting that $J_{-n}=(-1)^n\J$ and $H_{-n}^{(1)}=(-1)^n\H$, we only prove for $n\geq 1$. Using the uniform asymptotic expansions \eqref{eq:Modi-Bessel:1}--\eqref{eq:Modi-Bessel:2} and the relations \eqref{eq:Bessel:relation}, we get
\begin{equation*}
\frac{\H(k\thR)}{\J(k\thR)} = -2\i e^{-n\pi\i}e^{-2n\xi(w)} \frac{\left(1+\eta_2(n,w)\right)\left(1+\eta_1(n,\infty)\right)}{1+\eta_1(n,w)}
\end{equation*}
where $w=\frac{k\siz L}{n}-\i \frac{k\hR}{n}$. Noting from \eqref{assumption:1} that $\siz L\geq \hR$, which implies $\abs{\arg w}\leq \frac{\pi}{4}$. From \eqref{eq:Modi-Bessel:err:eta}, there holds
\begin{equation*}
\abs{\frac{\H(k\thR)}{\J(k\thR)}} \ls e^{-2n\Re\left(\xi(w)\right)} \frac{\big(1+\hat{M}(n,w)\big)\big(1+\hat{M}(n,\infty)\big)}{\big|1-\hat{M}(n,w)\big|},\ n\geq 1.
\end{equation*}
Applying \eqref{eq:Modi-Bessel:err:M} and \eqref{assumption:1}, 
\[
M(\infty)=\frac {1}{12}+\frac{1}{6\sqrt{5}}\leq 0.16,~M(w)\leq 0.16+\frac{\hR}{\siz L}\frac{n}{k\siz L}\leq 0.16+\frac{n}{10}.
\]
and hence for $n\geq 1$,
\begin{align}
&0< \hat{M}(n,\infty)\leq \frac{0.32}{n}\exp\left(\frac{0.32}{n}\right)\leq 0.45<1, \label{Minfty}\\
&0< \hat{M}(n,w)\leq \left(\frac{0.32}{n}+\frac 15\right)\exp\left(\frac{0.32}{n}+\frac 15\right)\leq 0.88<1.
\end{align}
These lead to
\begin{equation}\label{eq:stab:22}
\abs{\frac{\H(k\thR)}{\J(k\thR)}} \ls e^{-2n\Re\left(\xi(w)\right)},\ n\geq 1.
\end{equation}
In addition, from the Cauchy-Schwarz inequality,
\begin{equation}\label{eq:stab:22:1}
\abs{\inthR \J(k\tt)f_n(t)t dt} \leq \left(\inthR t\abs{\J(k\tt)}^2 dt\right)^{\frac 12}\left(\inthR t\abs{f_n(t)}^2 dt\right)^{\frac 12}
\end{equation}
To analyze $\J(k\tt)$, we denote by $\bR = \frac{\siz}{\siz-1}R$. Since from \eqref{assumption:1} that $L\geq \frac 2{\siz-1}R$, there holds $\bR\leq R+\frac L2\leq\hR$. We use the splitting
\begin{equation*}
\inthR t\abs{\J(k\tt)}^2 dt \leq \int_0^{\bR} t\abs{\J(k\tt)}^2 dt + \int_{\bR}^{\hR} t\abs{\J(k\tt)}^2 dt =: P_1+P_2.
\end{equation*}
For the first part $P_1$, from \eqref{eq:Bessel:bound}, we get
\begin{equation}\label{eq:stab:24}
P_1 \leq \frac{R^2}{2} + \int_R^{\bR} te^{2k\siz (t-R)} dt \ls 1 + \hR^2 e^{2k\siz (\bR-R)}\ls e^{k\siz L}.
\end{equation}
Next, we turn to estimate $P_2$. Let $W(t)=\frac{k\si_0(t-R)}{n}-\i \frac{kt}{n}$. Clearly, $W(\hR)=w$ and 
\[
\abs{\arg W}\leq\arctan\frac{\bR}{\si_0(\bR-R)}=\frac{\pi}{4},\quad n^2\abs{1+W^2}\ge n^2\big(\Re(W)\big)^2\overset{\eqref{assumption:1}}{\gtrsim} 1,\quad \forall t\in [\bR,\hR].
\]
From \eqref{eq:Bessel:relation}, Lemma~\ref{lem:prop:Modi-Bessel}, and \eqref{Minfty}, we have, for $\bR\leq t\leq\hR$,
\begin{equation*}
\abs{\J(k\tt)} = \abs{I_n(nW)} \leq \frac{1}{\sqrt{2n\pi}} \frac{e^{n\Re(\xi(W))}}{\abs{1+W^2}^{1/4}} \frac{\big|1+\hat{M}(n,W)\big|}{\big|1-\hat{M}(n,\infty)\big|} \ls e^{n\Re(\xi(W))}\ls e^{n\Re(\xi(w))},
\end{equation*}
and hence,
\begin{equation*}
P_2 = \int_{\bR}^{\hR} t\abs{\J(k\tt)}^2 dt \ls e^{2n\Re(\xi(w))},
\end{equation*}
which together with \eqref{eq:stab:24} gives
\begin{equation*}
\inthR t\abs{\J(k\tt)}^2 dt \ls e^{k\siz L}+e^{2n\Re(\xi(w))}.
\end{equation*}
By inserting the above estimate, \eqref{eq:stab:22:1}, and \eqref{eq:stab:22} into \eqref{eq:hC} we get
\begin{equation*}
\abs{\hC_n} \ls e^{-2n\Re\left(\xi(w)\right)}\left(e^{\frac 12 k\siz L}+e^{n\Re\left(\xi(w)\right)} \right)\left(\inthR r\abs{f_n(r)}^2 dr\right)^{\frac 12}.
\end{equation*}
From  Lemma~\ref{lem:prop:Modi-Bessel} (i) and noting that $n^2< k^2R^2+\lam k\hR^2\overset{\eqref{assumption:1}}{\leq} k^2\siz^2L^2$, there holds
\begin{equation*}
-n\Re(\xi(w))\leq \frac{n^2}{2k\siz L}-k\siz L\leq -\frac{1}{2} k\siz L,
\end{equation*}
which leads to
\begin{equation}
\abs{\hC_n} \ls e^{-\frac 12 k\siz L}\left(\inthR r\abs{ f_n(r)}^2 dr\right)^{\frac 12}.
\end{equation}
Since $\abs{\J(kr)}\leq 1$, we get
\begin{equation}\label{eq:stab:26}
\intR r\abs{p_n^{(3)}(r)}^2 dr \ls e^{-k\siz L}\inthR r\abs{f_n(r)}^2 dr.
\end{equation}
The combination of \eqref{eq:stab:19}, \eqref{eq:stab:20}, and \eqref{eq:stab:26} leads to
\begin{equation}\label{eq:stab:27}
\intR r\abs{\hu_n(r)}^2 dr \ls \left(\frac {1}{k^2} + e^{-k\siz L}\right) \inthR r\abs{f_n(r)}^2 dr \ls \frac {1}{k^2}\inthR r\abs{f_n(r)}^2 dr,
\end{equation}
where we have used $k^2\siz^2 L^2e^{-k\siz L}\ls 1$ and $\siz L\gtrsim 1$ to derive the last inequality.

By plugging \eqref{eq:stab:27} into \eqref{eq:stab:15} we conclude that \eqref{eq:stab:12} holds for Case 2.
 
\textbf{Case 3}: $n=0$.
For $0\leq r<R$, using the same splitting \eqref{eq:stab:18}, $\hu_0$ can be written as follows,
\begin{equation*}
\hu_0(r) = p_0^{(1)}(r)+p_0^{(2)}(r)+p_0^{(3)}(r).
\end{equation*}
Analogously, the estimates of \eqref{eq:stab:19} and \eqref{eq:stab:20} still hold for $n=0$, that is,
\begin{equation}\label{eq:stab:29}
\intR r\abs{p_0^{(j)}(r)}^2 dr \ls \frac{1}{k^2} \inthR r\abs{f_0(r)}^2 dr,\quad j=1,2.
\end{equation} 
For the term $p_0^{(3)}(r)=\hC_0 J_0(kr)$, it follows from \eqref{eq:Bessel:zero} and \eqref{eq:Bessel:zero:bound} that
\begin{equation}\label{eq:stab:30}
\abs{H_0^{(j)}(k\tt)} \eqsim \abs{k\tt}^{-\frac 12}e^{(-1)^j k\siz (t-R)},\quad R\leq t \leq \hR,\quad j=1,2. 
\end{equation}
Noting that $\Jv=\frac 12 (\Hv+H_\nu^{(2)})$, there holds
\begin{equation}\label{eq:stab:31}
\abs{\Jz(k\tt)}  \eqsim \abs{k\tt}^{-\frac 12}e^{k\siz (t-R)},\ R\leq t \leq \hR. 
\end{equation}
Then we get
\begin{equation}\label{eq:stab:32}
\abs{\frac{\Hz(k\thR)}{\Jz(k\thR)}} \eqsim e^{-2k\siz L},
\end{equation}
and 
\begin{equation}\label{eq:stab:33:1}
\intRhR t\abs{\Jz(k\tt)}^2 dt \ls \frac{1}{k}\intRhR \frac{t}{\abs{\tt}}e^{2k\siz (t-R)} dt\ls \frac{1}{k^2\siz}e^{2k\siz L}.
\end{equation}
Besides, since $\abs{\Jz(kt)}\leq 1$, we have
\begin{equation}\label{eq:stab:33}
\inthR t\abs{\Jz(k\tt)}^2 dt \ls R^2+\frac{1}{k^2\siz}e^{2k\siz L} \ls e^{2k\siz L}. 
\end{equation}
By inserting \eqref{eq:stab:32} and \eqref{eq:stab:33} into $\hC_0$ given by \eqref{eq:hC} we get
\begin{equation}\label{eq:stab:34}
\abs{\hC_0} \ls e^{-k\siz L}\left(\inthR r\abs{f_0(r)}^2 dr\right)^{\frac 12}
\end{equation}
This leads to
\begin{equation*}
\intR r\abs{p_0^{(3)}(r)}^2 dr \ls e^{-2k\siz L}\inthR r\abs{f_0(r)}^2 dr,
\end{equation*}
which together with \eqref{eq:stab:29} gives
\begin{equation}\label{eq:stab:35}
\intR r\abs{\hu_0(r)}^2 dr \ls \left(\frac{1}{k^2} + e^{-2k\siz L}\right) \inthR r\abs{f_0(r)}^2 dr \ls \frac{1}{k^2}\inthR r\abs{f_0(r)}^2 dr.
\end{equation}
Since \eqref{eq:stab:15} still holds for $n=0$, \eqref{eq:stab:12} follows by plugging \eqref{eq:stab:35} into \eqref{eq:stab:15}. This completes the proof of Theorem~\ref{thm:stab:L2} for the 2D case.


\subsubsection{Proof of  Theorem~\ref{thm:stab:L2} for 3D case}
Similar to the 2D case, it suffices to prove that
\begin{equation}\label{eq:stab:3d:0}
\inthR r^2\abs{\hu_\l^m(r)}^2 dr \ls{} \frac{1}{k^2} \inthR r^2\abs{f_\l^m(r)}^2 dr,\quad \text{for}~\l\in\N,~m\in\Z\cap[-\l,\l].
\end{equation}
By comparing \eqref{eq:hulm}--\eqref{eq:hC2} with \eqref{eq:hun}--\eqref{eq:hC}, similar to Cases 1--2 of the proof of \eqref{eq:stab:12},  we have, for $(\l+\frac 12)^2\ge 1$,
\begin{align*}
\inthR r\abs{\tr^{\frac12}\hu_\l^m(r)}^2 dr \ls \frac{1}{k^2}\inthR r\abs{\be^{-\frac12}r^\frac12 f_\l^m(r)}^2 dr,
\end{align*}
which implies \eqref{eq:stab:3d:0} for $\l\ge 1$.
 
It remains to prove for $\l=m=0$. From \eqref{eq:hulm},
\begin{equation}\label{eq:lzero-1}
\begin{aligned}
\tr^{\frac12}\hu_0^0(r) = &~\frac{\pi\i}{2}J_{\frac 12}(k\tr)\int_r^{\hR}H_{\frac 12}^{(1)}(k\tt)\big(\be^{-\frac12}t^\frac12 f_0^0(t)\big)tdt \\
 &~+ \frac{\pi\i}{2}H_{\frac 12}^{(1)}(k\tr)\intr J_{\frac 12}(k\tt)\big(\be^{-\frac12}t^\frac12 f_0^0(t)\big)tdt + \hC_{\frac 12}J_{\frac 12}(k\tr).
\end{aligned}
\end{equation}
Noting from \eqref{eq:Modi-Bessel:zero} that, we have, 
\begin{equation}\label{eq:stab:3d:3}
\abs{J_{\frac 12}(kr)}\eqsim \abs{H_{\frac 12}^{(1)}(kr)}\eqsim \frac 1{k^\frac12 r^{\frac 12}},\quad \text{for $0<r<R$},
\end{equation}
and
\begin{equation}\label{eq:stab:3d:4}
\abs{J_{\frac 12}(k\tr)}\eqsim\frac{e^{k\siz(r-R)}}{k^{\frac12}\abs{\tr}^{\frac12}},~
\abs{H_{\frac 12}^{(1)}(k\tr)}\eqsim\frac{e^{-k\siz(r-R)}}{k^{\frac12}\abs{\tr}^{\frac12}},\quad \text{for $R\leq r\leq\hR$}.
\end{equation}
Therefore
\begin{equation}\label{eq:stab:3d:5}
\inthR r\abs{J_{\frac12}(k\tr)}^2 dr \ls \intR r\abs{J_{\frac12}(kr)}^2 dr+\intRhR r\abs{J_{\frac12}(k\tr)}^2 dr \ls \frac{e^{2k\siz L}}{k},
\end{equation}
and from \eqref{eq:hC2},
\begin{equation}\label{eq:stab:3d:6}
\begin{aligned}
\abs{\hC_{\frac12}} &{}\ls e^{-2k\siz L}\left(\inthR r\abs{J_{\frac12}(k\tr)}^2 dr\right)^{1/2}\left(\inthR r^2\abs{f_0^0(r)}^2 dr\right)^{1/2} \\
 &{}\ls \frac{e^{-k\siz L}}{\sqrt{k}}\left(\inthR r^2\abs{f_0^0(r)}^2 dr\right)^{1/2}.
\end{aligned}
\end{equation}
Combining \eqref{eq:stab:3d:5} and \eqref{eq:stab:3d:6}, we get
\begin{equation}\label{eq:lzero-2}
\inthR r\abs{\hC_{\frac12}J_{\frac12}(k\tr)}^2 dr\ls \frac{1}{k^2}\inthR r^2\abs{f_0^0(r)}^2 dr.
\end{equation}
In addition, by using \eqref{eq:stab:3d:3} and \eqref{eq:stab:3d:4}, we have
\begin{equation}\label{eq:lzero-3}
\begin{aligned}
 &\inthR r\abs{J_{\frac12}(k\tr)\intrhR H_{\frac12}^{(1)}(k\tt)\big(\be^{-\frac12}t^\frac12 f_0^0(t)\big)t dt}^2 dr \\
\ls &~\inthR\intrhR r\abs{J_{\frac12}(k\tr)}^2  t \abs{H_{\frac12}^{(1)}(k\tt)}^2 dtdr\inthR r^2\abs{f_0^0(r)}^2 dr \\
\ls &~\inthR\intrhR \frac{1}{k^2} dtdr\inthR r^2\abs{f_0^0(r)}^2 dr\ls\frac{1}{k^2}\inthR r^2\abs{f_0^0(r)}^2 dr,
\end{aligned}
\end{equation}
and similarly,
\begin{equation}\label{eq:lzero-4}
\inthR r\abs{H_{\frac12}^{(1)}(k\tr)\intr J_{\frac12}(k\tt)\big(\be^{-\frac12}t^\frac12 f_0^0(t)\big)t dt}^2 dr\ls \frac{1}{k^2}\inthR r^2\abs{f_0^0(r)}^2 dr.
\end{equation}
From \eqref{eq:lzero-1} and \eqref{eq:lzero-2}--\eqref{eq:lzero-4} we get \eqref{eq:stab:3d:0} for $\l=m=0$. The proof of Theorem~\ref{thm:stab:L2} is completed. \hfill\QED

\subsection{$H^1$-estimates and inf--sup condition}
From \eqref{eq:PML:Var:1}, \eqref{eq:enorm} and \eqref{eq:stab:L2}, we have
\begin{equation*}
\He{\hu}^2 = \Re\big(a(\hu,\hu)\big)+2k^2\Lt{\hu}{\D}^2 \ls \abs{\inD{f}{\hu}} +2k^2\Lt{\hu}{\D}^2 \ls \Lt{f}{\D}^2.
\end{equation*}
That is, the following $H^1$-stability estimates holds:
\begin{corollary}\label{cor:H1-stab}
Under the conditions of Theorem~\ref{thm:stab:L2}, there holds
\begin{equation}\label{eq:H1-stab}
\He{\hu}\ls \Lt{f}{\D}.
\end{equation}
\end{corollary}

Next we consider the inf--sup condition of the truncated PML problem. 
\begin{theorem}\label{thm:mu} Under the conditions of Theorem~\ref{thm:stab:L2},  the following inf--sup condition of the truncated PML problem \eqref{eq:PML:all}--\eqref{eq:PML:bound:all} holds 
\begin{equation}\label{eq:inf-sup} 
\mu := \inf_{0\neq u\in H_0^1(\D)}\sup_{0\neq v\in H_0^1(\D)} \frac{\abs{a(u,v)}}{\He{u}\He{v}} \eqsim k^{-1}.
\end{equation}
\end{theorem}
\begin{proof}
We first derive the lower bound for the inf--sup constant $\mu$. For every $0\neq u\in H_0^1(\D)$, take $g=2k^2u\in L^2(\D)$, and consider the following dual problem to \eqref{eq:PML:Var:1}: Find $z\in H_0^1(\D)$ such that
\begin{equation}\label{eq:dual:Var}
a(v,z) = \inD{v}{g},\quad \forall v\in H_0^1(\D).
\end{equation}
 Similar to the $H^1$-stability \eqref{eq:H1-stab}, we have
\begin{equation}\label{eq:inf-sup:1}
\He{z}\ls \Lt{g}{\D} \ls k^2\Lt{u}{\D} \ls k\He{u},
\end{equation}
 then
\begin{equation}\label{eq:inf-sup:2}
\He{u+z}\ls k\He{u}.
\end{equation}
On the other hand, from \eqref{eq:enorm} and \eqref{eq:dual:Var},
\begin{equation}\label{eq:inf-sup:3} 
\Re(a(u,u+z)) = \He{u}^2-2k^2\Lt{u}{\D}^2 + \Re(a(u,z)) = \He{u}^2.
\end{equation}
The combination of \eqref{eq:inf-sup:2}  and \eqref{eq:inf-sup:3} leads to the explicit lower bound on $\mu$:
\begin{equation}\label{eq:inf-sup:lower}
\mu \gtrsim \frac{1}{k}
\end{equation}

By following \cite{chandler2008wave}, the upper bounds on the inf--sup constant can be obtained by giving an example. Note that, for every $0\neq u\in H_0^1(\D)$, there holds
\begin{equation}\label{eq:inf-sup:upper}
\mu \leq \sup\limits_{0\neq v\in H_0^1(\D)}\frac{\abs{a(u,v)}}{\He{u}\He{v}}.
\end{equation}
Denote by $x=(x_1,\cdots, x_d)$, and define $w(x)=\abs{x}^2-R^2$ and 
\begin{equation*}
u(x)=\left\{
\begin{aligned}
& e^{\i kx_1}w(x),  & x\in\Om,\\
& 0,  & \mbox{otherwise}.
\end{aligned}
\right.
\end{equation*}
We prove only for 2D case since the proofs for 1D and 3D case follow almost the same procedure. For $d=2$, it's easy to verify that $\Lt{u}{\D}=\Lt{w}{\Om} = \sqrt{\pi/3}R^3$. By integrating by parts, we have
\begin{equation*}
\abs{a(u,v)} = \abs{\int_{\D}A\na u\cdot\na\bar{v}-Bk^2u\bar{v} dx} = \abs{\int_{\Om} (\Delta u+k^2u)\bar{v}dx},\quad \forall v\in H_0^1(\D).
\end{equation*}
A simple calculation leads to  $\Delta u+k^2u = 4e^{\i kx_1}(1+\i k x_1)$, then
\begin{align*}
\abs{a(u,v)} \leq 4\left(\int_{\Om}(1+k^2x_1^2)dx_1dx_2\right)^{\frac 12}\Lt{v}{\Om} \leq 2\sqrt{\pi}R(2+kR)\Lt{v}{\Om},
\end{align*}
which implies
\begin{align*}
\frac{\abs{a(u,v)}}{\He{u}\He{v}} \leq \frac{2\sqrt{\pi}R(2+kR)\Lt{v}{\Om}}{k^2\Lt{u}{\D}\Lt{v}{\D}} \leq \frac{2\sqrt{3}}{kR}+\frac{4\sqrt{3}}{k^2R^2},\quad \forall~0\neq v\in H_0^1(\D).
\end{align*}
From \eqref{eq:inf-sup:upper} we arrive at the following explicit upper bound on $\mu$:
\begin{equation}\label{eq:inf-sup:upper:b}
\mu \ls \frac{1}{k}+\frac{1}{k^2}.
\end{equation}

A combination of \eqref{eq:inf-sup:lower} and \eqref{eq:inf-sup:upper:b} yields \eqref{eq:inf-sup}. This completes the proof of the theorem.
\end{proof}

\subsection{Convergence of the truncated PML solution}\label{sec:err:trunc}
In this subsection we suppose that $\supp f\subset\Om$ and then give an estimate of $u-\hu$ in $\Om$
as an application of the inf--sup condition. 

 For any domain $G\subset\R^d$, denote by
\begin{align}\label{aG}
a_G(u,v) &:= (A\na u, \na v)_G-k^2(Bu,v)_G.
\end{align}
Introduce the energy norm on $H^1(G)$:
\begin{align*}
\He{v}_G := \Big(\Re(a_G(v,v))+2k^2\Lt{v}{G}^2\Big)^{\frac 12}.
\end{align*}
Clearly, $a=a_\D$ and $\He{\cdot}=\He{\cdot}_\D$ (see \eqref{a} and \eqref{eq:enorm}). Let $\Om^c:=\R^d\setminus\bar\Om.$
We introduce the following extension operators. $P,~P^*: H^\frac12(\Ga)\mapsto H^1(\Om^c)$ are defined by
\begin{align}\label{P}
a_{\Om^c}(P w,\vp) = a_{\Om^c}(\vp,P^* w)=0,\quad P w|_\Ga=P^* w|_\Ga=w,\quad \forall \vp\in H^1_0(\Om^c).
\end{align} 
 $\hat P, \hat P^*: H^\frac12(\Ga)\mapsto \{v\in H^1(\hOm):v|_{\hat{\Gamma}}=0\}$ are defined by
\begin{align}\label{hP}
a_{\hOm}(\hat P w,\psi)= a_{\hOm}(\psi,\hat P^* w)=0,\quad \hat P w|_\Ga=\hat P^* w|_\Ga=w,\quad \forall \psi\in H^1_0(\hOm).
\end{align} 
We remark that, as consequences of Lemma~\ref{LaG} below,   the above four extension operators are well-defined.
Note that $\tu(r,\theta) = u(\tr,\theta)$ where $u$ is the solution to \eqref{eq:Helm}--\eqref{eq:Somm}. It is easy to verify that $\tu=u$ in $\bar\Om$ and $\tu=P u$ in $\Om^c$ and that
\begin{align}\label{tu}
a_\Om(\tu,v)+a_{\Om^c}(\tu,v)=(f,v)_\Om,\quad\forall v\in H^1_0(\R^d).
\end{align} On the other hand, we have $\hu=\hat P\hu$ in $\hOm$.

The following continuity and coercivity estimates of the sesquilinear form $a_G$  will be used in the convergence analysis.
\begin{lemma}\label{LaG} Under the conditions of Theorem~\ref{thm:stab:L2} there hold
\begin{align}
\abs{a_{G}(u,v)}&\ls \He{u}_{G}\He{v}_{G} \quad\forall u,~v \in H^1(G),~G \subset\R^d \label{eq:continuity:d}\\
\abs{a_{G}(u,v)}&\ls k\Ho{u}{G}\He{v}_{G} \quad\forall u,~v \in H^1(G),~G \subset\R^d \label{eq:continuity:a}\\
\abs{a_{\Om^c\setminus\hOm}(v,v)}&\gtrsim \He{v}_{\Om^c\setminus\hOm}^2 \quad\forall v \in H^1(\Om^c\setminus\hOm),\label{eq:coercivity:a}\\
\abs{a_{\hOm}(v,v)}&\gtrsim k^{-1}\He{v}_{\hOm}^2 \quad\forall v \in H^1(\hOm), v|_\Ga=0.\label{eq:coercivity:b} 
\end{align}
\end{lemma}
\begin{proof}The first two continuity estimates follow from the Cauchy-Schwarz inequality and the definitions of the norms and the sesquilinear forms. We omit the details. Next we prove the two coercivity estimates only for two dimensions. The 3D case is similar and the 1D case is simpler. From \eqref{def:medium} and \eqref{aG} (cf. \eqref{eq:stab:7}), we have
\begin{align*}
\Re\big(a_{\Om^c\setminus\hOm}(v,v)\big) &= \int_0^{2\pi}\int_{\hR}^\infty\bigg(\frac{1+\si\de}{1+\si ^2}r\abs{v_r}^2 + \frac{1+\si\de}{1+\de ^2}\frac{1}{r}\abs{v_\theta}^2 + (\si\de-1)k^2 r\abs{v}^2\bigg).
\end{align*}
Since $\si\de=\frac{\siz^2(r-R)}{r}\ge \siz\frac{\siz L}{L+R}\ge\siz\sqrt{3}\ge 3$ for $r\ge\hR$, \eqref{eq:coercivity:a} holds. It remains to prove \eqref{eq:coercivity:b}. We have
\begin{align*}
\Re\big(a_{\hOm}(v,v)\big) &= \int_0^{2\pi}\int_R^{\hR}\bigg(\frac{1+\si\de}{1+\si ^2}r\abs{v_r}^2 + \frac{1+\si\de}{1+\de ^2}\frac{1}{r}\abs{v_\theta}^2 + (\si\de-1)k^2 r\abs{v}^2\bigg),\\
\Im\big(a_{\hOm}(v,v)\big) &= \int_0^{2\pi}\int_R^{\hR}\bigg(\frac{\de-\si}{1+\si ^2}r\abs{v_r}^2 + \frac{\si-\de}{1+\de ^2}\frac{1}{r}\abs{v_\theta}^2 - (\si+\de)k^2 r\abs{v}^2\bigg).
\end{align*}
Let $\lam>0$ and $\check R\in (R,\hR)$ be constants satisfying
\begin{align}\label{eq:checkR}
\Big(\frac{1}{\siz}+\siz\Big)\siz\Big(1-\frac{R}{\check R}\Big)\ge 2\lam,\quad 3\ln\frac{\check R}{R}\le\frac12, \quad \frac{1}{1+\siz^2}\ge (\check R^2-R^2)k^2\lam.
\end{align}
As a matter of fact, it's easy to verify that $\lam=\frac{1}{6k \check R}$ and $\check R=R+\frac{1}{3k(1+\siz^2)}$ satisfy the above conditions. 
Then
\begin{align*}
&(1+\lam)\Re\big(a_{\hOm}(v,v)\big)-\si_0^{-1}\Im\big(a_{\hOm}(v,v)\big)\\
=~&\lam\int_0^{2\pi}\int_R^{\hR}\bigg(\frac{1+\si\de}{1+\si ^2}r\abs{v_r}^2 + \frac{1+\si\de}{1+\de ^2}\frac{1}{r}\abs{v_\theta}^2 + (\si\de+1)k^2 r\abs{v}^2\bigg)\\
&+\int_0^{2\pi}\int_R^{\hR}\bigg(\frac{2+\si\de-\frac{\de}{\si}}{1+\si ^2}r\abs{v_r}^2 + \frac{(\frac{1}{\si}+\si)\de}{1+\de ^2}\frac{1}{r}\abs{v_\theta}^2 + \Big(\Big(\frac{1}{\si}+\si\Big)\de-2\lam\Big)k^2 r\abs{v}^2\bigg)\\
\ge~& \lam\He{v}_{\hOm}^2+\int_0^{2\pi}\int_R^{\check R}\bigg(\frac{2}{1+\si ^2}r\abs{v_r}^2-2\lam k^2r\abs{v}^2\bigg)
\end{align*}
where we have used the first inequality of \eqref{eq:checkR}. Noting that, for $R\le r\le \check R$,
\begin{align*}
r\abs{v}^2=\frac{1}{r}\int_R^r\frac{\pa}{\pa s}(s^2|v(s)|^2)ds\le \frac{3}{r}\int_R^{\check R}r|v|^2dr+r\int_R^{\check R} r|v_r|^2dr,
\end{align*}
we have 
\begin{align*}
\int_R^{\check R}r|v|^2dr\le 3\ln\frac{\check R}{R}\int_R^{\check R}r|v|^2dr+\frac12(\check R^2-R^2)\int_R^{\check R} r|v_r|^2 dr
\end{align*}
which together with \eqref{eq:checkR} implies 
\begin{align*}
\int_R^{\check R}r|v|^2dr\le(\check R^2-R^2)\int_R^{\check R} r|v_r|^2dr
\end{align*}
Therefore, 
\begin{align*}
&(1+\lam)\Re\big(a_{\hOm}(v,v)\big)-\si_0^{-1}\Im\big(a_{\hOm}(v,v)\big)\ge \lam\He{v}_{\hOm}^2,
\end{align*}
which implies \eqref{eq:coercivity:b}. This completes the proof of the lemma.
\end{proof}
\begin{theorem}\label{thm:err:trunc} Suppose $\supp f \subset \Omega$.
Under the conditions of Theorem~\ref{thm:stab:L2}, the truncated PML problem \eqref{eq:PML:Var:1} has a unique solution $\hu\in H_0^1(\D)$. Furthermore, there holds:
\begin{align}
\He{u-\hu}_\Om &\ls ke^{-2k\siz L}\Gnorm{u},  &d=1,\label{eq:err:trunc:1d}\\
\He{u-\hu}_\Om &\ls k^5 e^{-2 k\siz L\big(1-\frac{R^2}{\hR^2+\siz^2L^2}\big)^{1/2}}\Gnorm{u}, &d=2,3.\label{eq:err:trunc}
\end{align}
\end{theorem}
\begin{proof} The 1D case \eqref{eq:err:trunc:1d} can be proved easily by noting that $\hu-u = D_1 e^{-\i kx} + D_2 e^{\i kx}$ in $\Om$. The details are omitted.
Next we turn to prove \eqref{eq:err:trunc}. Write $\xi:=u-\hu$ in $\bar\Om$ and $\xi:=\hat P u-\hu$ in $\hOm$. We have, for any $v\in H_0^1(D)$,
\begin{align*}
a(\xi, v)=&~a_\Om(u,v)+a_{\hOm} (\hat P u,v)-(f,v)_\Om & (\text{from } \eqref{eq:PML:Var:1})\\
=&~a_{\hOm}(\hat P u,v)-a_{\hOm}(P u,v)& (\text{from } \eqref{tu})\\
=&~a_{\hOm}(\hat P u-P u,\hat P^* v)& (\text{since } v-\hat P^* v\in H_0^1(\hOm))\\
=&~a_{\hOm}(\hat P u-P u,\hat P^* v)&\\
&~-a_{\hOm}(\hat P u-P u, P^* v)+a_{\Om^c\setminus\hOm}(P u, P^* v)& (\text{since } (\hat P u-P u)|_\Ga=0)\\
=&~a_{\hOm}(P u-\hat P u, P^* v-\hat P^* v)+a_{\Om^c\setminus\hOm}(P u, P^* v).&
\end{align*}
which together with \eqref{eq:continuity:d} implies that
\begin{equation}\label{e1}
\abs{a(\xi, v)}\ls \he{P u-\hat P u}_{\hOm}\he{P^* v-\hat P^* v}_{\hOm}+\He{P u}_{\Om^c\setminus\hOm}\He{P^* v}_{\Om^c\setminus\hOm}.
\end{equation}
Next we estimate the terms on the right hand side. Let $w:=P u-\hat P u$ in $\hOm$ and let $\vp\in H^1(\hOm)$ satisfy:
\begin{equation*}
\vp|_{\hGamma}=w|_{\hGamma}=Pu|_{\hGamma}, \quad \vp|_\Ga=0,\quad\text{ and } \Ho{\vp}{\hOm}\ls\hGnorm{Pu}. 
\end{equation*}
Clearly, $w-\vp\in H_0^1(\hOm)$. It follows from \eqref{eq:continuity:a}, \eqref{eq:coercivity:b}, \eqref{P}, and \eqref{hP} that
\begin{align*}
\He{w-\vp}_{\hOm}^2\ls k a_{\hOm}(w-\vp,w-\vp)=k a_{\hOm}(-\vp,w-\vp)\ls k^2\Ho{\vp}{\hOm}\He{w-\vp}_{\hOm}
\end{align*}
which implies that
\begin{align}\label{e2}
\he{P u-\hat P u}_{\hOm}\ls  k^2\Ho{\vp}{\hOm}\ls k^2\hGnorm{Pu}.
\end{align}
Similarly,
\begin{align}\label{e3}
\he{P^* v-\hat P^* v}_{\hOm}\ls k^2\hGnorm{P^* v},
\end{align}
and
\begin{align}\label{e4}
\He{P u}_{\Om^c\setminus\hOm}\ls k\hGnorm{Pu},\quad\He{P^* v}_{\Om^c\setminus\hOm}\ls k\hGnorm{P^* v}.
\end{align}
Moreover, from \cite[(2.31)]{chen2005} and noting \cite[10.47.5]{olver2010nist} for 3D case, we have
\begin{align}\label{e5}
\hGnorm{P v}, ~\hGnorm{P^* v}\leq e^{-k\siz L\big(1-\frac{R^2}{\hR^2+\siz^2L^2}\big)^\frac12}\Gnorm{v},\;\forall v\in H^\frac12(\Ga).
\end{align}
By plugging \eqref{e2}--\eqref{e5} into \eqref{e1} we conclude that
\begin{align*}
\abs{a(\xi, v)}\ls k^4e^{-2k\siz L\big(1-\frac{R^2}{\hR^2+\siz^2L^2}\big)^\frac12}\Gnorm{u}\Gnorm{v}.
\end{align*}
Therefore, it follows from the trace theorem and Theorem~\ref{thm:mu} that
\begin{align*}
\He{\xi}\ls k^5e^{-2k\siz L\big(1-\frac{R^2}{\hR^2+\siz^2L^2}\big)^\frac12}\Gnorm{u},
\end{align*}
which implies \eqref{eq:err:trunc}. This completes the proof of the theorem.
\end{proof}

\begin{remark}
A similar procedure as for  \cite[Theorem 2.6]{chen2005} gives
\begin{equation*}
\Ho{u-\hu}{\Om} \leq \widetilde{C}k^2e^{-k\siz L\left(1-\frac{R^2}{\hR^2+\siz^2 L^2}\right)^{1/2}}\Gnorm{\hu},\quad d=2
\end{equation*}
where $\widetilde{C}$ polynomially depends on $k$. Note that our convergence rate in \eqref{eq:err:trunc} is twice that of the above estimate.
\end{remark}
\subsection{$H^2$-estimates}
In this subsection, we will show the $H^2$-regularity for the truncated PML problem \eqref{eq:PML:all}--\eqref{eq:PML:bound:all} as an immediate consequence of Theorem~\ref{thm:stab:L2}. Since $A$ and $B$ are discontinuous on $\Gamma$, we define the space $H^2(\OchO):=\{v : v|_{\Om}\in H^2(\Om), v|_{\hOm}\in H^2(\hOm)\}$ with the semi-norm $\sHt{v}{\OchO}:=(\sHt{v}{\Om}^2+\sHt{v}{\hOm}^2)^{1/2}$.
The following regularity estimates hold:
\begin{corollary}\label{cor:H2-stab}
Under the conditions of Theorem~\ref{thm:stab:L2}, we have
\begin{equation}\label{eq:H2-stab}
\sHt{\hu}{\OchO} \ls k\Lt{f}{\D}.
\end{equation}
\end{corollary}
\begin{proof}
From \eqref{eq:PML:all}, there holds
\[-\na\cdot(A\na \hu) = Bk^2\hu+f,
\]
Applying the estimates in \cite[Theorem 4.5]{huang2007uniform} and using \eqref{eq:stab:L2}, we obtain
\[ \sHt{\hu}{\Om} + \siz\sHt{\hu}{\hOm} \ls \Lt{Bk^2\hu+f}{\D}\ls k\Lt{f}{\D}.
\]
\end{proof}

In the next section, we apply the regularity estimate to derive preasymptotic error estimates for the CIP-FEM.

\section{CIP-FEM and its preasymptotic error analysis}\label{sec:CIP-FEM}
In this section, we first introduce the CIP-FEM for the truncated PML problem \eqref{eq:PML:Var:1}, then give a preasymptotic error analysis of it. We suppose that $\supp f\subset\Om$ in this section.

\subsection{CIP-FEM}
Let $\Mh$ be a curvilinear triangulation of $\D$ (cf. \cite{melenk2010convergence,melenk2011wavenumber,monk2003finite}). For any $K\in\Mh$, we define $h_K:=\diam(K)$ and $h:=\max_{K\in \Mh} h_K$. Similarly, for each edge/face $e$ of $K\in\Mh$, we define $h_e:=\diam(e)$. Assume that $h_K\eqsim h_e \eqsim h$, and $\mathring{K}\cap\Gamma =\emptyset$. Additionally, denote by $\widehat{K}$ the reference element and $F_K$ the element maps from $\widehat{K}$ to $K\in \Mh$, which satisfy the Assumption 5.2 in \cite{melenk2010convergence}, that is, $F_K$ can be written as $F_K=R_K\circ A_K$, where $A_K$ is an affine map and the maps $R_K$ and $A_K$ satisfy for constants $C_{\text{affine}},~C_{\text{metric}},~\bar{\gamma}>0$ independent of $h$:
\begin{align*}
&\Vert A'_K\Vert_{L^\infty(\widehat K)}\leq C_{\text{affine}}h,\qquad \Vert (A'_K)^{-1}\Vert_{L^\infty(\widehat K)}\leq C_{\text{affine}}h^{-1}, \\
&\Vert (R'_K)^{-1}\Vert_{L^\infty(\widetilde K)}\leq C_{\text{metric}},\qquad \Vert \na^n R_K\Vert_{L^\infty(\widetilde K)}\leq C_{\text{affine}}\bar{\gamma}^n n!\qquad \forall n\in\N_0.
\end{align*}
Here, $\widetilde{K}=A_K(\widehat{K})$. Let $\EhI$ be the set of edges/faces of $\Mh$ in $\Om$. For every $e=\partial K_1 \cap \partial K_2\in\EhI$, we define the jump $[v]$ of $v$ on $e$ as follows:
\begin{equation*}
[v]|_e:=v|_{K_1}-v|_{K_2}.
\end{equation*}
Now we introduce the energy space $V$ and the sesquilinear form $\ah (\cdot,\cdot)$ on $V\times V$. Firstly, denote by $H^2(\Mh):=\{v: v|_K\in H^2(K), \forall K\in\Mh\}$ with the semi-discrete norm $\sHt{v}{\Mh}:=\big(\sum_{K\in\Mh}\sHt{v}{K}\big)^{1/2}$. Then, we define
\begin{align*}
V &:= H_0^1(\D)\cap H^2(\Mh) \\
\ah (u,v) &:= a(u,v)+J(u,v)  \\
J(u,v) &:= \sum\limits_{e\in\EhI} \gamma_e h_e\ine{[\na u\cdot n]}{[\na v\cdot n]}
\end{align*}
where penalty parameters $\gamma_e,e\in\EhI$ are numbers with nonpositive imaginary parts. 

It is clear that, if $\hu\in H^2(\OchO)$ is the solution to \eqref{eq:PML:all}--\eqref{eq:PML:bound:all}, then $J(\hu,v) = 0$ for any $v\in V$, and from \eqref{eq:PML:Var:1} we have
\begin{equation}\label{eq:PML:Var}
\ah (\hu,v) = \inOm{f}{v}, \quad \forall v\in V.
\end{equation}

Let $V_h$ be the linear finite element approximation space 
\begin{equation*}
V_h:=\{v_h\in H_0^1(\D):v_h|_K\circ F_K \in \mathcal{P}_1(\widehat{K}), \forall K\in\Mh\}.
\end{equation*}
where $\mathcal{P}_1(\widehat{K})$ denotes the set of all first order polynomials on $\widehat{K}$. Then the CIP-FEM reads as: Find $u_h \in V_h$ such that
\begin{equation}\label{eq:CIP-FEM}
\ah(u_h,v_h) = \inOm{f}{v_h} \quad\forall v_h\in V_h.
\end{equation}
\begin{remark}
(i) Clearly, if $\gamma_e \equiv 0$, then the CIP-FEM becomes FEM. 

(ii) The CIP-FEM was first introduced by Douglas and Dupont \cite{Douglas1976Interior} for second order elliptic and parabolic PDEs, and it was applied to the the Helmholtz problem \eqref{eq:Helm} with the impedance boundary condition by Wu, Zhu and Du \cite{Wu2013Pre,zhu2013preasymptotic,du2015preasymptotic}. 

(iii) Note that the sesquilinear form $a(\cdot,\cdot)$ is coercive in the PML region $\hOm$ (cf. Lemma~\ref{LaG}) and hence the truncated PML problem behaves more like an elliptic one. Based on this consideration, we only introduce penalty terms in $J(u,v)$ for edges/faces in $\Om$ in order to reduce the pollution error.

(iv) In this paper we consider the scattering problem with time dependence $e^{-\i\omega t}$, that is, the sign before $\i$ in \eqref{eq:Somm} is negative. If we consider the scattering problem with time dependence $e^{\i\omega t}$, that is, the sign before $\i$ in  \eqref{eq:Somm} is positive, then the penalty parameters should be complex numbers with  nonnegative imaginary parts.
\end{remark}

\subsection{Elliptic projections} We define the elliptic projection operators as follows (cf. \cite{zhu2013preasymptotic}). Introduce the sesquilinear forms
\begin{equation}\label{eq:err:bh}
b(u,v)=\inD{A\na u}{\na v} ~\text{and}~ \bh (u,v) = b(u,v)+J(u,v)
\end{equation}
 For any $w\in H_0^1(\D)\cap H^2(\OchO)$, define its elliptic projections $\Ph w\in V_h$ as:
\begin{equation}\label{eq:ell-proj}
\begin{aligned}
\bh (\Pha w,v_h) &{}= \bh (w,v_h),\quad\forall v_h\in V_h, \\
\bh (v_h,\Phm w) &{}= \bh (v_h,w),\quad\forall v_h\in V_h.
\end{aligned}
\end{equation}
In this subsection we derive  error estimates of the elliptic projections.  For simplicity, we assume that $\gamma_e\equiv\gamma$.

First, we prove the following continuity and coercivity properties for $\bh$.
\begin{lemma}\label{lem:cont-coer}
There exists a constant $\gamma_0>0$ such that, if $\Re\gamma\geq -\gamma_0$ and $\abs{\gamma}\ls 1$, then for any $u,v \in V$ and $v_h\in V_h$, there holds  
\begin{align*}
\abs{\bh (u,v)} &\ls \left(\sHo{u}{\D}+h\sHt{u}{\Mh}\right)\left(\sHo{v}{\D}+h\sHt{v}{\Mh}\right),\\
\Re(\bh (v_h,v_h))&\gtrsim \sHo{v_h}{\D}^2.
\end{align*}
\end{lemma}
\begin{proof}
First,  by using the local trace inequality $\Lt{v}{\pa K}\ls h_K^{-\frac12}\Lt{v}{K}+h_K^{\frac12}\Lt{\na v}{K}$ (see \cite{chen2010selected}), we get
\begin{align*}
\abs{J(u,v)}&{}\leq\abs{\gamma} \bigg(\sum\limits_{K\in\Mh\cap\Om} h_K\Lt{\na u\cdot n}{\pa K}^2\bigg)^\frac12\bigg(\sum\limits_{K\in\Mh\cap\Om} h_K\Lt{\na v\cdot n}{\pa K}^2\bigg)^\frac12\\
&{}\ls \left(\sHo{u}{\D}+h\sHt{u}{\Mh}\right)\left(\sHo{v}{\D}+h\sHt{v}{\Mh}\right).
\end{align*}
which together with \eqref{eq:err:bh} implies the continuity of $\bh $.

Secondly, from the definition of the coefficient $A$ in \S~\ref{ss:PML} and a similar analysis as above, we conclude that
\begin{align*}
\Re(b(v_h,v_h))\geq C_b\sHo{v_h}{\D}^2 \text{ and }  \sum\limits_{e\in\EhI} h_e\Lt{[\na v_h\cdot n]}{e}^2\le \be_0\sHo{v_h}{\D},
 \;\forall v_h\in V_h,
\end{align*}
 where $C_b$ and $\be_0$ are  positive constants. Then \begin{equation*}
\Re(\bh (v_h,v_h))=\Re(b(v_h,v_h))+\Re\gamma \sum\limits_{e\in\EhI} h_e\Lt{[\na v_h\cdot n]}{e}^2 \geq \frac{C_b}{2}\sHo{v_h}{\D}^2
\end{equation*}
if $\Re\gamma\ge -\gamma_0$ with $\gamma_0:=\frac{C_b}{2\be_0}.$
This completes the proof of the lemma.
\end{proof}

The following lemma gives error estimates of the elliptic projections.
\begin{lemma}\label{lem:err:ell-proj}
Under the conditions of Lemma~\ref{lem:cont-coer}, for any $w \in H_0^1(\D)\cap H^2(\OchO)$, there holds 
\begin{equation*}
\Lt{w-\Ph w}{\D} + h\sHo{w-\Ph w}{\D} + h^2\sHt{w-\Ph w}{\Mh} \ls h^2\sHt{w}{\OchO}.
\end{equation*}
\end{lemma}
\begin{proof}
We prove only the estimates for $\Phm w$ since the proof for $\Pha w$ follows almost the same procedure. Clearly,
\begin{equation}\label{eq:ell-proj:orth}
 \bh(v_h,w-\Phm w) = 0,\quad\forall v_h\in V_h.
\end{equation}
Let $I_h$ be the finite element interpolation operator onto $V_h$. We have the following interpolation error estimates (cf. \cite{brenner2007mathematical,lenoir1986optimal,bernardi1989optimal}). For any $ v\in H^2(\OchO)$,
\begin{equation}\label{eq:err:intp}
\Lt{v-\Ih v}{\D}+h\sHo{v-\Ih v}{\D}+h^2\sHt{v-\Ih v}{\Mh} \ls h^2\sHt{v}{\OchO}.
\end{equation}
From Lemma~\ref{lem:cont-coer} and \eqref{eq:ell-proj:orth} and \eqref{eq:err:intp},
\begin{align*}
\sHo{\Ih w-\Phm w}{\D}^2 &{}\ls \Re(\bh (\Ih w-\Phm w,\Ih w-\Phm w))\\
 &{}\ls \Re(\bh (\Ih w-\Phm w,\Ih w-w)) \\
 &{}\ls \sHo{\Ih w-\Phm w}{\D}\left(\sHo{\Ih w -w}{\D}+h\sHt{w}{\OchO}\right) \\
 &{}\ls h\sHo{\Ih w-\Phm w}{\D}\sHt{w}{\OchO},
\end{align*}
which together with \eqref{eq:err:intp} implies that
\begin{equation}\label{eq:err:ell-proj:h1}
\sHo{w-\Phm w}{\D}\ls h\sHt{w}{\OchO}.
\end{equation}
 
To estimate the $L^2$-error, we consider the following auxiliary problem:
\begin{equation}\label{eq:err:10}
-\na\cdot(A\na z) = w-\Phm w\quad \mbox{in }\D,\qquad z = 0 \quad\mbox{on }\hat{\Gamma}.
\end{equation}
It can be shown that (cf. \cite{huang2007uniform})
\begin{equation}\label{eq:err:11}
\sHt{z}{\OchO}\ls\Lt{w-\Phm w}{\D}.
\end{equation}
Testing \eqref{eq:err:10} by the conjugate of $w-\Phm w$, applying integration by parts and using  Lemma~\ref{lem:cont-coer} and \eqref{eq:err:intp} and \eqref{eq:err:ell-proj:h1}, then we get
\begin{align*}
\Lt{w-\Phm w}{\D}^2 &{}=\bh (z,w-\Phm w)=\bh (z-\Ih z,w-\Phm w) \\
 &{}\ls h^2\sHt{z}{\OchO}\sHt{w}{\OchO} 
\end{align*}
which together with \eqref{eq:err:11} gives
\[
\Lt{w-\Phm w}{\D} \ls h^2\sHt{w}{\OchO}.
\]

Finally, we can complete the proof of the lemma by using the inverse estimates (see \cite{chen2010selected}) and \eqref{eq:err:intp} and \eqref{eq:err:ell-proj:h1},
\begin{align*}
\sHt{w-\Phm w}{\Mh} &{}\ls \sHt{w-\Ih w}{\Mh} + h^{-1}\sHo{\Ih w-\Phm w}{\D} \ls \sHt{w}{\OchO}.
\end{align*}
\end{proof}

\subsection{Preasymptotic error estimates} 
In this subsection we prove  preasymptotic error estimates for the CIP-FE discretization \eqref{eq:CIP-FEM} of the truncated PML problem \eqref{eq:PML:all}--\eqref{eq:PML:bound:all} by using the modified duality argument proposed in \cite{zhu2013preasymptotic}. 
\begin{theorem}\label{thm:err}
Under the conditions of Theorem~\ref{thm:stab:L2} and Lemma~\ref{lem:cont-coer}, let $u_h$ denotes the CIP-FE solution of \eqref{eq:CIP-FEM}. Then there exists a constant $C_0>0$ independent of $k,h,f$ and the penalty parameters such that if 
\[
k^3h^2\leq C_0,
\]
then the following error estimates hold:
\begin{align}
\He{\hu-u_h} &\ls (kh+k^3h^2)\Lt{f}{\Om}, \label{eq:H1err}\\
\Lt{\hu-u_h}{\D}&\ls k^2h^2\Lt{f}{\Om}.\label{eq:L2err}
\end{align}
\end{theorem}
\begin{proof}
Denote by $e=\hu-u_h$. Introduce the dual problem: 
Find $w\in H_0^1(\D)$ such that
\begin{equation}\label{eq:dual}
a(v,w) = \inD{v}{e},\quad \forall v\in H_0^1(\D).
\end{equation}
Similar to the regularity \eqref{eq:H2-stab}, we have
\begin{equation}\label{eq:err:14}
\sHt{w}{\OchO}\ls k\Lt{e}{\D}.
\end{equation}
From \eqref{eq:PML:Var} and \eqref{eq:CIP-FEM}, the following $Galerkin~orthogonality$ holds:
\begin{equation}\label{eq:Gal-orth}
\ah (e,v_h)=0 \qquad \forall v_h\in V_h.
\end{equation}
It follows from \eqref{eq:dual}, \eqref{eq:Gal-orth}, and \eqref{eq:err:intp} and Lemmas~\ref{lem:cont-coer} and \ref{lem:err:ell-proj} that
\begin{align}\label{eq:DA}
\Lt{e}{\D}^2 &=\ah(e,w)=\ah(e,w-\Phm w) \\
 &=\bh(\hu-\Ih\hu,w-\Phm w) - k^2\inD{Be}{w-\Phm w} \notag\\
 &\ls h^2\sHt{\hu}{\OchO}\sHt{w}{\OchO}+k^2h^2 \Lt{e}{\D}\sHt{w}{\OchO}.\notag
\end{align}
From \eqref{eq:H2-stab} and \eqref{eq:err:14}, there holds
\[
\Lt{e}{\D}^2 \leq Ck^2h^2 \Lt{f}{\Om}\Lt{e}{\D}+Ck^3h^2\Lt{e}{\D}^2.
\]
Therefore there exists a constant $C_0>0$ such that if $k^3h^2\leq C_0$, then \eqref{eq:L2err} holds.

To show \eqref{eq:H1err}, we denote by $\eta=u_h-\Pha\hu\in V_h$. It follows from Lemma~\ref{lem:cont-coer} and \eqref{eq:ell-proj} that
\begin{align*}
\sHo{\eta}{\D}^2 &{}\ls\Re(\bh(\eta,\eta)) = \Re\left(\ah(\eta,\eta)+k^2\inD{B\eta}{\eta}\right)\\
 &{}= \Re\left(\ah(\hu-\Pha\hu,\eta)+k^2\inD{B\eta}{\eta}\right) \\
 &{}= \Re\left(-k^2\inD{B(\hu-u_h)}{\eta}\right) \\
 &{}\ls k^2\Lt{e}{\D}\Lt{\eta}{\D}.
\end{align*}
Noting that $\Lt{\eta}{\D}\leq\Lt{\hu-\Pha\hu}{\D}+\Lt{e}{\D}$, and using Lemma~\ref{lem:err:ell-proj} and \eqref{eq:L2err}, we have 
\begin{align*}
\sHo{\eta}{\D} &\ls k^3h^2\Lt{f}{\Om},
\end{align*}
which implies that
\begin{align*}
\sHo{e}{\D} &\ls \sHo{\hu-\Pha\hu}{\D}+\sHo{\eta}{\D} \ls (kh+k^3h^2)\Lt{f}{\Om}.
\end{align*}
Then \eqref{eq:H1err} follows by noting that $\He{e}\eqsim \sHo{e}{\D}+k\Lt{e}{\D}.$ This completes the proof of the theorem.
\end{proof}

\begin{remark} (i) The traditional duality argument using $I_hw$ in the step \eqref{eq:DA} instead of $\Phm w$ gives only error estimates under the mesh condition that $k^2h$ is small enough (see \cite{aziz1979scattering,ihlenburg1997finite, melenk2010convergence,melenk2011wavenumber}).

(ii) For the Helmholtz problem \eqref{eq:Helm} with the impedance boundary condition, the same preasymptotic error estimates were obtained by  
 \cite{Wu2013Pre,zhu2013preasymptotic,du2015preasymptotic}.
 
(iii) The error bound in \eqref{eq:H1err} consists of two terms. The first term $O(kh)$ is of the same order as the interpolation error. For large wave number problems, the second term $O(k^3h^2)$ may be large even if the interpolation error is small. The second term is called the pollution error \cite{bs00}.

(iv) Since the CIP-FEM reduces to FEM when the penalty parameter $\gamma=0$, this theorem holds also for the FEM.

(v) The estimates of $u-u_h$ may be obtained by combining Theorems~\ref{thm:err:trunc} and \ref{thm:err}.

(vi) The penalty parameter may be tuned to reduce the pollution error (see the next section).  

(vii) It is possible to show that the CIP-FEM is absolute stable if the penalty parameters have negative imaginary parts (cf. \cite{Wu2013Pre,zhu2013preasymptotic}). This will be explored in a future work. 
\end{remark}

By combining Theorems~\ref{thm:stab:L2}and  \ref{thm:err} and Corollary~\ref{cor:H1-stab}, we obtain the following stability estimates for CIP-FE solution.
\begin{corollary}\label{cor:stab:CIP-FEM}
Under the conditions of Theorem~\ref{thm:err}, there holds,
\begin{equation}\label{eq:stab:CIP-FEM}
\He{u_h}+k\Lt{u_h}{\D} \ls \Lt{f}{\Om}.
\end{equation}
and hence the CIP-FEM is well-posed.
\end{corollary}

\section{Numerical results}\label{sec:Experimental}
In this section, we simulate the Helmholtz problem \eqref{eq:Helm}--\eqref{eq:Somm} with $\Om=\{x\in\R^2:\abs{x}< 1\}$ and the following $f$ and exact solution.
\begin{equation}\label{eq:numer:ue}
f=\left\{
\begin{aligned}
&1,\quad\mbox{in }\Om, \\
&0,\quad\mbox{otherwise}.
\end{aligned}
\right.\qquad
 u = \begin{cases}
 \frac{\i\pi}{2k}H_1^{(1)}(k)J_0(kr)-\frac{1}{k^2},&\text{in }\Om,\\
 \frac{\i\pi }{2k} J_1(k) H_0^{(1)}(kr),&\text{otherwise}.
 \end{cases}
\end{equation}
The problem is first truncated by the PML and then discretized by the CIP-FEM \eqref{eq:CIP-FEM}. We will report some numerical results on the FEM (i.e. CIP-FEM with $\gamma_e\equiv 0$) and the CIP-FEM with the following penalty parameters which are obtained by a dispersion analysis for two dimensional problems on equilateral triangulations:
\begin{equation}\label{eq:penalty}
\gamma_e = -\frac{\sqrt{3}}{24}-\frac{\sqrt{3}}{1728}(kh)^2.
\end{equation} 
The codes are written in MATLAB. Since the mesh generation program usually produces triangulations in which most triangle elements are approximate equilateral triangles (see the left graph of Figure~\ref{fig:mesh_ue}), it is expected that the above choice of penalty parameters can reduce the pollution error. Set $k\geq 2,~\siz =5$ and $L=R=1$ and hence \eqref{eq:assumption} holds.  
From Theorems~\ref{thm:err} and \ref{thm:err:trunc}, the $H^1$-error of the FE or CIP-FE solution $u_h$ is bounded by
\begin{equation}\label{eq:exp:1}
\Ho{u-u_h}{\Om}\le C_1 kh+C_2 k^3h^2+C_3k^5e^{-9.8k}.
\end{equation}
for some constants $C_j~(j=1,2,3)$ under the condition of $k^3h^2\le C_0$. The first term on the right hand side of \eqref{eq:exp:1} corresponds to the interpolation error, the second term to the pollution error, and the third term to the PML truncation error. 

The right graph of Figure~\ref{fig:mesh_ue} plots the traces of the real parts of the exact solution, FE solution, and CIP-FE solution as $y=0$ and $x$ from $0$ to $1$ when $k=100$ on a mesh with $h\approx\frac 1{128}$. It is obvious that the CIP-FE solution fits the exact solution much better than FE solution.

\begin{figure}[htbp]
\centering
\begin{minipage}[c]{0.5\textwidth}
\centering
\includegraphics[height=5cm,width=6.25cm]{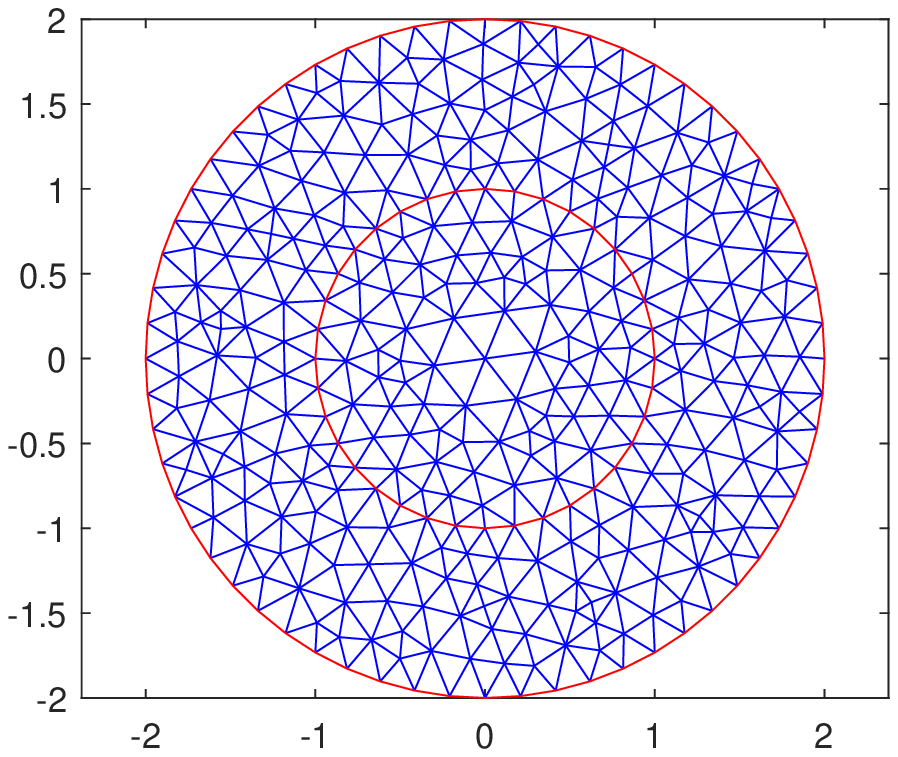}
\end{minipage}%
\begin{minipage}[c]{0.5\textwidth}
\centering
\includegraphics[height=5cm,width=6.25cm]{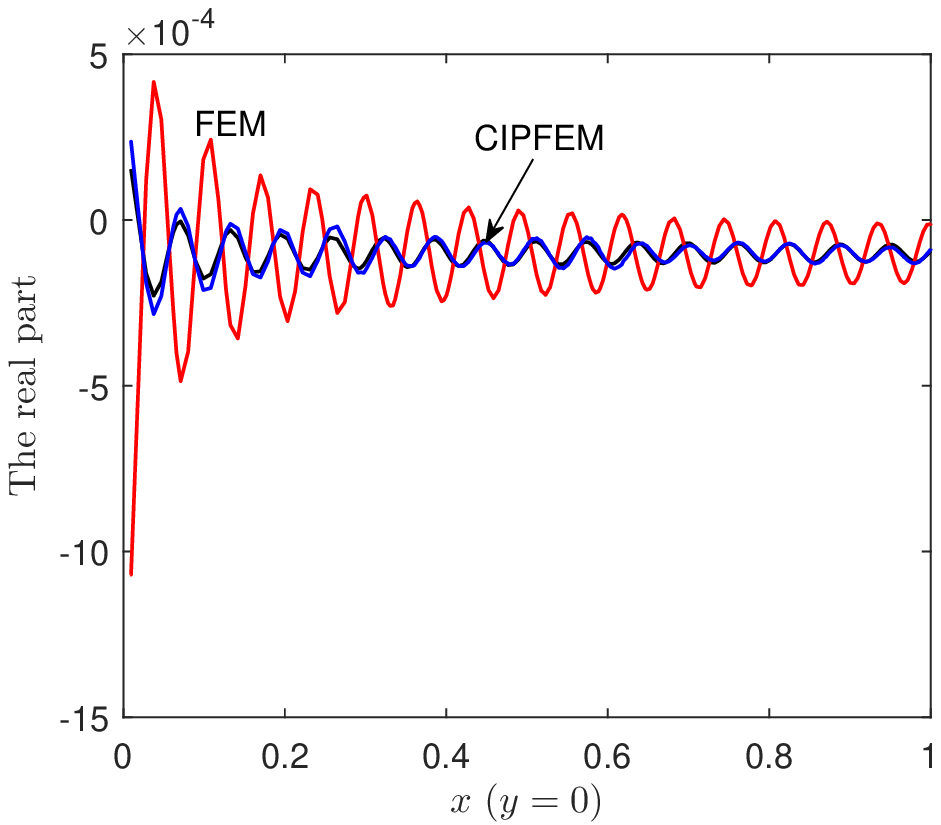}
\end{minipage}
\caption{Left graph: A sample mesh. Right graph: The traces of the real parts of the exact solution (black), FE solution (red), and CIP-FE solution (blue) when $k=100$ and $h\approx\frac 1{128}$.}
\label{fig:mesh_ue}
\end{figure}

Figure~\ref{fig:ex3err} plots the relative errors in $H^1$-norm of the (CIP-)FE solutions and FE interpolations for $k=5,~25$, and $100$, respectively. It is shown that, for $k=5$ the errors of FE and CIP-FE solutions fit those of the corresponding FE interpolation very well, which implies the pollution errors do not show up for small wave number. For large $k$, the relative errors of the FE solutions decay slowly on a range starting with a point far from the decaying point of the corresponding FE interpolations. This behavior show clearly the effect of the pollution error of FEM. The CIP-FE solutions behave similarly as the FE solutions but the pollution range of the former is much smaller than that of the later. 
\begin{figure}[h]
\centering
\begin{minipage}[c]{0.5\textwidth}
\centering
\includegraphics[height=5cm,width=6.25cm]{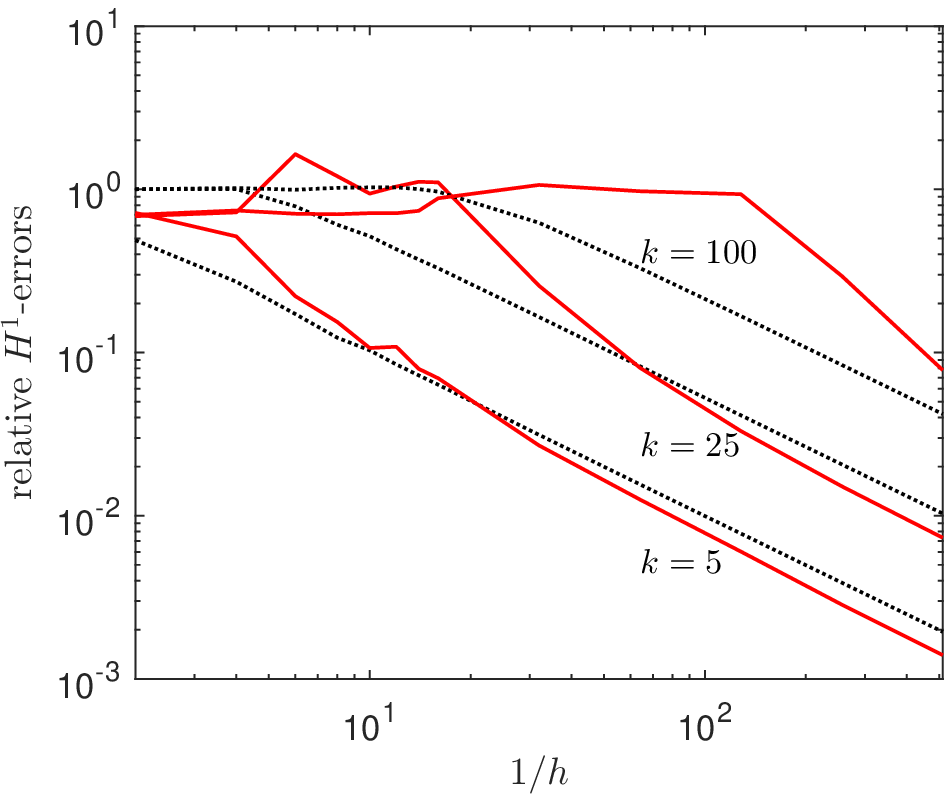}
\end{minipage}%
\begin{minipage}[c]{0.5\textwidth}
\centering
\includegraphics[height=5cm,width=6.25cm]{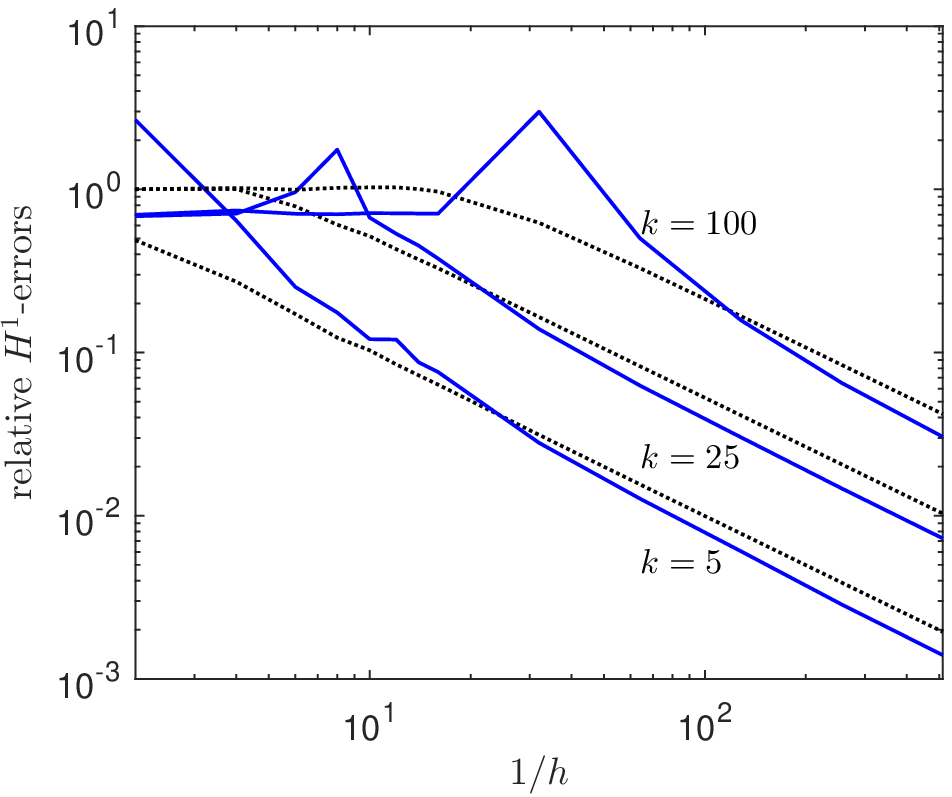}
\end{minipage}
\caption{The relative $H^1$-errors of the FE solution (left graph) and the CIP-FE solution (right graph), compared with the relative $H^1$-error of the FE interpolation (dotted) for $k = 5, 25$, and $100$, respectively.}
\label{fig:ex3err}
\end{figure}

More intuitively, we fix $kh=1$ and $kh=0.5$, and plot the relative $H^1$-errors of both the FEM and CIP-FEM for increasing wave numbers $k$ in one figure (see Figure~\ref{fig:ex6kh}). It's obvious that the effect of the pollution error of FEM works when $k$ becomes greater than some value less than 50, while the pollution error of CIP-FEM is almost invisible for $k$ up to 250. In addition, the errors of the FE interpolation and CIP-FE solution for the case of $kh=1$ are almost twice as big as those for the case of $kh=0.5$.

\begin{figure}[htbp]
\centering
\includegraphics[width=8cm]{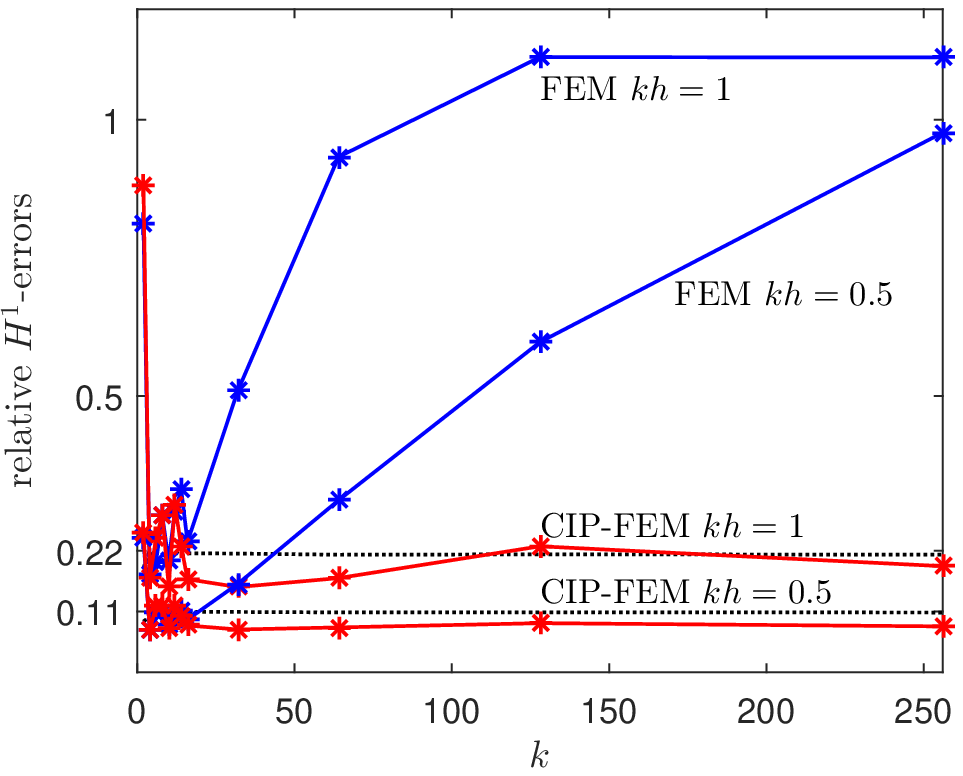}
\caption{The relative $H^1$-errors of the FE solution, the CIP-FE solution and the FE interpolation (dotted), with the mesh size which satisfies $kh=1$ and $kh=0.5$, respectively.}
\label{fig:ex6kh}
\end{figure}

\section*{Acknowledgment} The authors would like to thank two anonymous referees for their insightful and constructive comments and suggestions that have helped us improve 
the paper essentially.

\end{document}